\documentclass[11pt]{article}
\usepackage{epsf}
\usepackage{amsmath}
\usepackage{epsfig}
\usepackage{times}
\usepackage{amssymb}
\usepackage{amsthm}
\usepackage{setspace}
\usepackage{cite}

\usepackage{algorithmic}  
\usepackage{algorithm}

\usepackage{shadow}
\usepackage{fancybox}
\usepackage{fancyhdr}

\def\y{{\bf y}}

\def\x{{\bf x}}

\def\x{{\mathbf x}}

\def\x{{\bf x}}
\def\y{{\bf y}}

\def\q{{\bf q}}

\def\h{{\bf h}}

\def\be{\begin{equation}}
\def\ee{\end{equation}}
\def\ba{\left[\begin{array}}
\def\ea{\end{array}\right]}

\def\x{{\bf x}}
\def\y{{\bf y}}

\def\q{{\bf q}}

\def\1{{\bf 1}}

\def\g{{\bf g}}
\def\0{{\bf 0}}

\newtheorem{theorem}{Theorem}

\newtheorem{lemma}{Lemma}

\begin{document}

\begin{singlespace}

\title {Another look at the Gardner problem 
}
\author{
\textsc{Mihailo Stojnic}
\\
\\
{School of Industrial Engineering}\\
{Purdue University, West Lafayette, IN 47907} \\
{e-mail: {\tt mstojnic@purdue.edu}} }
\date{}
\maketitle

\centerline{{\bf Abstract}} \vspace*{0.1in}

In this paper we revisit one of the classical perceptron problems from the neural networks and statistical physics. In \cite{Gar88} Gardner presented a neat statistical physics type of approach for analyzing what is now typically referred to as the Gardner problem. The problem amounts to discovering a statistical behavior of a spherical perceptron. Among various quantities \cite{Gar88} determined the so-called storage capacity of the corresponding neural network and analyzed its deviations as various perceptron parameters change. In a more recent work \cite{SchTir02,SchTir03} many of the findings of \cite{Gar88} (obtained on the grounds of the statistical mechanics replica approach) were proven to be mathematically correct. In this paper, we take another look at the Gardner problem and provide a simple alternative framework for its analysis. As a result we reprove many of now known facts and rigorously reestablish a few other results.

\vspace*{0.25in} \noindent {\bf Index Terms: Gardner problem; storage capacity}.

\end{singlespace}

\section{Introduction}
\label{sec:back}

In this paper we will revisit a classical perceptron type of problem from neural networks and statistical physics/mechanics. A great deal of the problem's popularity has its roots in a nice work \cite{Gar88}. Hence, to describe the problem we will closely follow what was done in \cite{Gar88}. We start with the following dynamics:
\begin{equation}
H_{ik}^{(t+1)}=\mbox{sign} (\sum_{j=1,j\neq k}^{n}H_{ij}^{(t)}X_{jk}-T_{ik}).\label{eq:defdyn}
\end{equation}
Following \cite{Gar88} for any fixed $1\leq i\leq m$ we will call each $H_{ij},1 \leq j\leq n $, the icing spin, i.e. $H_{ij}\in\{-1,1\},\forall i,j$. Following \cite{Gar88} further we will call $X_{jk},1\leq j\leq n$, the interaction strength for the bond from site $j$ to site $i$. $T_{ik},1\leq i\leq m,1\leq k\leq n$, will be the threshold for site $k$ in pattern $i$ (we will typically assume that $T_{ik}=0$; however all the results we present below can be modified easily so that they include scenarios where $T_{ik}\neq 0$).

Now, the dynamics presented in (\ref{eq:defdyn}) works by moving from a $t$ to $t+1$ and so on (of course one assumes an initial configuration for say $t=0$). Moreover, the above dynamics will have a fixed point if say there are strengths $X_{jk},1\leq j\leq n,1\leq k\leq m$, such that for any $1\leq i\leq m$
\begin{eqnarray}
& & H_{ik}\mbox{sign} (\sum_{j=1,j\neq k}^{n}H_{ij}X_{jk}-T_{ik})=1\nonumber \\
& \Leftrightarrow & H_{ik}(\sum_{j=1,j\neq k}^{n}H_{ij}X_{jk}-T_{ik})>0,1\leq j\leq n,1\leq k\leq n.\label{eq:defdynfp}
\end{eqnarray}
Now, of course this is a well known property of a very general class of dynamics. In other words, unless one specifies the interaction strengths the generality of the problem essentially makes it easy. In \cite{Gar88} then proceeded and considered the spherical restrictions on $X$. To be more specific the restrictions considered in \cite{Gar88} amount to the following constraints
\begin{equation}
\sum_{j=1}^{n}X_{ji}^2=1,1\leq i\leq n.\label{eq:cosntX}
\end{equation}
Then the fundamental question that was considered in \cite{Gar88} is the so-called storage capacity of the above dynamics or alternatively a neural network that it would represent. Namely, one then asks how many patterns $m$ ($i$-th pattern being $H_{ij},1\leq j\leq n$) one can store so that there is an assurance that they are stored in a stable way. Moreover, since having patterns being fixed points of the above introduced dynamics is not enough to insure having a finite basin of attraction one often may impose a bit stronger threshold condition
\begin{eqnarray}
& & H_{ik}\mbox{sign} (\sum_{j=1,j\neq k}^{n}H_{ij}X_{jk}-T_{ik})=1\nonumber \\
& \Leftrightarrow & H_{ik}(\sum_{j=1,j\neq k}^{n}H_{ij}X_{jk}-T_{ik})>\kappa,1\leq j\leq n,1\leq k\leq n,\label{eq:defdynfpstr}
\end{eqnarray}
where typically $\kappa$ is a positive number.

In \cite{Gar88} a replica type of approach was designed and based on it a characterization of the storage capacity was presented. Before showing what exactly such a characterization looks like we will first formally define it. Namely, throughout the paper we will assume the so-called linear regime, i.e. we will consider the so-called \emph{linear} scenario where the length and the number of different patterns, $n$ and $m$, respectively are large but proportional to each other. Moreover, we will denote the proportionality ratio by $\alpha$ (where $\alpha$ obviously is a constant independent of $n$) and will set
\begin{equation}
m=\alpha n.\label{eq:defmnalpha}
\end{equation}
Now, assuming that $H_{ij},1\leq i\leq m,1\leq j\leq n$, are i.i.d. symmetric Bernoulli random variables, \cite{Gar88} using the replica approach gave the following estimate for $\alpha$ so that (\ref{eq:defdynfpstr}) holds with overwhelming probability (under overwhelming probability we will in this paper assume a probability that is no more than a number exponentially decaying in $n$ away from $1$)
\begin{equation}
\alpha_c(\kappa)=(\frac{1}{\sqrt{2\pi}}\int_{-\kappa}^{\infty}(z+\kappa)^2e^{-\frac{z^2}{2}}dz)^{-1}.\label{eq:garstorcap}
\end{equation}
Based on the above characterization one then has that $\alpha_c$ achieves its maximum over positive $\kappa$'s as $\kappa\rightarrow 0$. One in fact easily then has
\begin{equation}
\lim_{\kappa\rightarrow 0}\alpha_c(\kappa)=2.\label{eq:garstorcapk0}
\end{equation}
The result given in (\ref{eq:garstorcapk0}) is of course well known and has been rigorously established either as a pure mathematical fact or even in the context of neural networks and pattern recognition \cite{Schlafli,Cover65,Winder,Winder61,Wendel62,Cameron60,Joseph60,BalVen87,Ven86}. In a more recent work \cite{SchTir02,SchTir03} the authors also considered the Gardner problem and established that (\ref{eq:garstorcap}) also holds.

Of course, a whole lot more is known about the model (or its different variations) that we described above and will study here. All of our results will of course easily translate to these various scenarios. Instead of mentioning all of these applications here we in this introductory paper chose to present the key components of our mechanism on the most basic (and probably most widely known) problem. All other applications we will present in several forthcoming papers.

Also, we should mentioned that many variants of the model that we study here are possible from a purely mathematical perspective. However, many of them have found applications in various other fields as well. For example, a great set of references that contains a collection of results related to various aspects of neural networks and their bio-applications is  \cite{AgiAnnBarCooTan13a,AgiAnnBarCooTan13b,AgiBarBarGalGueMoa12,AgiBarGalGueMoa12,AgiAstBarBurUgu12}.

As mentioned above, in this paper we will take another look at the above described storage capacity problem. We will provide a relatively simple alternative framework to characterize it. However, before proceeding further with the presentation we will just briefly sketch how the rest of the paper will be organized. In Section \ref{sec:uncorgard} we will present the main ideas behind the mechanism that we will use to study the storage capacity problem. This will be done in the so-called uncorrelated case, i.e. when no correlations are assumed among patterns. In the last part of Section \ref{sec:uncorgard}, namely, Subsection \ref{sec:negkappa} we will then present a few results related to a bit harder version of a mathematical problem arising in the analysis of the storage capacity. Namely, we will consider validity of fixed point inequalities (\ref{eq:defdynfp}) when $\kappa<0$. In Section \ref{sec:corgard} we will then show the corresponding results when the patterns are correlated in a certain way. Finally, in Section \ref{sec:conc} we will provide a few concluding remarks.

\section{Uncorrelated Gardner problem}
\label{sec:uncorgard}

In this section we look at the so-called uncorrelated case of the above described Gardner problem. In fact, such a case is precisely what we described in the previous section. Namely, we will assume that all patterns $H_{i,1:n},1\leq i\leq m$, are uncorrelated ($H_{i,1:n}$ stands for vector $[H_{i1},H_{i2},\dots,H_{in}]$). Now, to insure that we have the targeted problem stated clearly we restate it again. Let $\alpha=\frac{m}{n}$ and assume that $H$ is an $m\times n$ matrix with i.i.d. $\{-1,1\}$ Bernoulli entries. Then the question of interest is: assuming that $\|\x\|_2=1$, how large $\alpha$ can be so that the following system of linear inequalities is satisfied with overwhelming probability
\begin{equation}
H\x\geq \kappa.\label{eq:defprobucor}
\end{equation}
This of course is the same as if one asks how large $\alpha$ can be so that the following optimization problem is feasible with overwhelming probability
\begin{eqnarray}
& & H\x\geq \kappa\nonumber \\
& & \|\x\|_2=1.\label{eq:defprobucor1}
\end{eqnarray}
To see that (\ref{eq:defprobucor}) and (\ref{eq:defprobucor1}) indeed match the above described fixed point condition it is enough to observe that due to statistical symmetry one can assume $H_{i1}=1,1\leq i\leq m$. Also the constraints essentially decouple over the columns of $X$ (so one can then think of $\x$ in (\ref{eq:defprobucor}) and (\ref{eq:defprobucor1}) as one of the columns of $X$). Moreover, the dimension of $H$ in (\ref{eq:defprobucor}) and (\ref{eq:defprobucor1}) should be changed to $m\times (n-1)$; however, since we will consider a large $n$ scenario to make writing easier we keep dimension as $m\times n$.

Now, it is rather clear but we do mention that the overwhelming probability statement is taken with respect to the randomness of $H$. To analyze the feasibility of (\ref{eq:defprobucor1}) we will rely on a mechanism we recently developed for studying various optimization problems in \cite{StojnicRegRndDlt10}. Such a mechanism works for various types of randomness. However, the easiest way to present it is assuming that the underlying randomness is standard normal. So to fit the feasibility of (\ref{eq:defprobucor1}) into the framework of \cite{StojnicRegRndDlt10} we will need matrix $H$ to be comprised of i.i.d. standard normals. We will hence without loss of generality in the remainder of this section assume that elements of matrix $H$ are indeed i.i.d. standard normals (towards the end of the paper we will briefly mention why such an assumption changes nothing in the validity of the results; also, more on this topic can be found in e.g. \cite{StojnicHopBnds10,StojnicLiftStrSec13,StojnicRegRndDlt10} where we discussed it a bit further).

Now, going back to problem (\ref{eq:defprobucor1}), we first recognize that it can be rewritten as the following optimization problem
\begin{eqnarray}
\xi_n=\min_{\x} \max_{\lambda\geq 0} & &  \kappa\lambda^T\1- \lambda^T H\x \nonumber \\
\mbox{subject to} & & \|\lambda\|_2= 1\nonumber \\
& & \|\x\|_2=1,\label{eq:uncorminmax}
\end{eqnarray}
where $\1$ is an $m$-dimensional column vector of all $1$'s. Clearly, if $\xi_n\leq 0$ then (\ref{eq:defprobucor1}) is feasible. On the other hand, if $\xi_n>0$ then (\ref{eq:defprobucor1}) is not feasible. That basically means that if we can probabilistically characterize the sign of $\xi_n$ then we could have a way of determining $\alpha$ such that $\xi_n\leq 0$. Below, we provide a way that can be used to characterize $\xi_n$. We do so by relying on the strategy developed in \cite{StojnicRegRndDlt10,StojnicGorEx10} and ultimately on the following set of results from \cite{Gordon85,Gordon88}.
\begin{theorem}(\cite{Gordon88,Gordon85})
\label{thm:Gordonmesh1} Let $X_{ij}$ and $Y_{ij}$, $1\leq i\leq n,1\leq j\leq m$, be two centered Gaussian processes which satisfy the following inequalities for all choices of indices
\begin{enumerate}
\item $E(X_{ij}^2)=E(Y_{ij}^2)$
\item $E(X_{ij}X_{ik})\geq E(Y_{ij}Y_{ik})$
\item $E(X_{ij}X_{lk})\leq E(Y_{ij}Y_{lk}), i\neq l$.
\end{enumerate}
Then
\begin{equation*}
P(\bigcap_{i}\bigcup_{j}(X_{ij}\geq \lambda_{ij}))\leq P(\bigcap_{i}\bigcup_{j}(Y_{ij}\geq \lambda_{ij})).
\end{equation*}
\end{theorem}
The following, more simpler, version of the above theorem relates to the expected values.
\begin{theorem}(\cite{Gordon85,Gordon88})
\label{thm:Gordonmesh2} Let $X_{ij}$ and $Y_{ij}$, $1\leq i\leq n,1\leq j\leq m$, be two centered Gaussian processes which satisfy the following inequalities for all choices of indices
\begin{enumerate}
\item $E(X_{ij}^2)=E(Y_{ij}^2)$
\item $E(X_{ij}X_{ik})\geq E(Y_{ij}Y_{ik})$
\item $E(X_{ij}X_{lk})\leq E(Y_{ij}Y_{lk}), i\neq l$.
\end{enumerate}
Then
\begin{equation*}
E(\min_{i}\max_{j}(X_{ij}))\leq E(\min_i\max_j(Y_{ij})).
\end{equation*}
\end{theorem}


We will split the rest of the presentation in this section into two subsections. First we will provide a mechanism that can be used to characterize a lower bound on $\xi_n$. After that we will provide its a counterpart that can be used to characterize an upper bound on a quantity similar to $\xi_n$ which has the same sign as $\xi_n$.

\subsection{Lower-bounding $\xi_n$}
\label{sec:uncorgardlb}

We will make use of Theorem \ref{thm:Gordonmesh1} through the following lemma (the lemma is  of course a direct consequence of Theorem \ref{thm:Gordonmesh1} and in fact is fairly similar to Lemma 3.1 in \cite{Gordon88}, see also \cite{StojnicHopBnds10} for similar considerations).
\begin{lemma}
Let $H$ be an $m\times n$ matrix with i.i.d. standard normal components. Let $\g$ and $\h$ be $m\times 1$ and $n\times 1$ vectors, respectively, with i.i.d. standard normal components. Also, let $g$ be a standard normal random variable and let $\zeta_{\lambda}$ be a function of $\x$. Then
\begin{equation}
P(\min_{\|\x\|_2=1}\max_{\|\lambda\|_2=1,\lambda_i\geq 0}(-\lambda^T H\x+g-\zeta_{\lambda})\geq 0)\geq
P(\min_{\|\x\|_2=1}\max_{\|\lambda\|_2=1,\lambda_i\geq 0}(\g^T\lambda+\h^T\x-\zeta_{\lambda})\geq 0).\label{eq:negproblemma}
\end{equation}\label{lemma:negproblemma}
\end{lemma}
\begin{proof}
The proof is basically similar to the proof of Lemma 3.1 in \cite{Gordon88} as well as to the proof of Lemma 7 in \cite{StojnicHopBnds10}. The only difference is the structure of the allowed set of $\lambda$'s. Such a difference changes nothing structurally in the proof, though.
\end{proof}

Let $\zeta_{\lambda}=-\kappa\lambda^T\1+\epsilon_{5}^{(g)}\sqrt{n}+\xi_n^{(l)}$ with $\epsilon_{5}^{(g)}>0$ being an arbitrarily small constant independent of $n$. We will first look at the right-hand side of the inequality in (\ref{eq:negproblemma}). The following is then the probability of interest
\begin{equation}
P(\min_{\|\x\|_2=1}\max_{\|\lambda\|_2=1,\lambda_i\geq 0}(\g^T\lambda+\h^T\x+\kappa\lambda^T\1-\epsilon_{5}^{(g)}\sqrt{n})\geq \xi_n^{(l)}).\label{eq:negprobanal0}
\end{equation}
After solving the minimization over $\x$ and the maximization over $\lambda$ one obtains
\begin{equation}
\hspace{-.3in}P(\min_{\|\x\|_2=1}\max_{\|\lambda\|_2=1,\lambda_i\geq 0}(\g^T\lambda+\h^T\x+\kappa\lambda^T\1-\epsilon_{5}^{(g)}\sqrt{n})\geq \xi_n^{(l)})=P(\|(\g+\kappa\1)_+\|_2-\|\h_i\|_2-\epsilon_{5}^{(g)}\sqrt{n}\geq \xi_n^{(l)}),\label{eq:negprobanal1}
\end{equation}
where $(\g+\kappa\1)_+$ is $(\g+\kappa\1)$ vector with negative components replaced by zeros. Since $\h$ is a vector of $n$ i.i.d. standard normal variables it is rather trivial that
\begin{equation}
P(\|\h\|_2<(1+\epsilon_{1}^{(n)})\sqrt{n})\geq 1-e^{-\epsilon_{2}^{(n)} n},\label{eq:devh}
\end{equation}
where $\epsilon_{1}^{(n)}>0$ is an arbitrarily small constant and $\epsilon_{2}^{(n)}$ is a constant dependent on $\epsilon_{1}^{(n)}$ but independent of $n$. Along the same lines, since $\g$ is a vector of $m$ i.i.d. standard normal variables it easily follows that
\begin{equation}
E\sum_{i=1}^{n}(\max\{\g_i+\kappa,0\})^2=mf_{gar}(\kappa),\label{eq:efgar}
\end{equation}
where
\begin{equation}
f_{gar}(\kappa)=\frac{1}{\sqrt{2\pi}}\int_{-\kappa}^{\infty}(\g_i+\kappa)^2e^{-\frac{\g_i^2}{2}}d\g_i.\label{eq:fgar}
\end{equation}
One then easily also has
\begin{equation}
P\left (\sqrt{\sum_{i=1}^{n}(\max\{\g_i+\kappa,0\})^2}>(1-\epsilon_{1}^{(m)})\sqrt{mf_{gar}(\kappa)}\right )\geq 1-e^{-\epsilon_{2}^{(m)} m},\label{eq:devg}
\end{equation}
where $\epsilon_{1}^{(m)}>0$ is an arbitrarily small constant and analogously as above $\epsilon_{2}^{(m)}$ is a constant dependent on $\epsilon_{1}^{(m)}$ and
$f_{gar}(\kappa)$ but independent of $n$. Then a combination of (\ref{eq:negprobanal1}), (\ref{eq:devh}), and (\ref{eq:devg}) gives
\begin{multline}
P(\min_{\|\x\|_2=1}\max_{\|\lambda\|_2=1,\lambda_i\geq 0}(\g^T\lambda+\h^T\x+\kappa\lambda^T\1-\epsilon_{5}^{(g)}\sqrt{n})\geq \xi_n^{(l)})\\\geq
(1-e^{-\epsilon_{2}^{(m)} m})(1-e^{-\epsilon_{2}^{(n)} n})
P((1-\epsilon_{1}^{(m)})\sqrt{mf_{gar}(\kappa)}-(1+\epsilon_{1}^{(n)})\sqrt{n}-\epsilon_{5}^{(g)}\sqrt{n}\geq \xi_n^{(l)}).
\label{eq:negprobanal2}
\end{multline}
If
\begin{eqnarray}
& & (1-\epsilon_{1}^{(m)})\sqrt{mf_{gar}(\kappa)}-(1+\epsilon_{1}^{(n)})\sqrt{n}-\epsilon_{5}^{(g)}\sqrt{n}>\xi_n^{(l)}\nonumber \\
& \Leftrightarrow & (1-\epsilon_{1}^{(m)})\sqrt{\alpha f_{gar}(\kappa)}-(1+\epsilon_{1}^{(n)})-\epsilon_{5}^{(g)}>\frac{\xi_n^{(l)}}{\sqrt{n}},\label{eq:negcondxipu}
\end{eqnarray}
one then has from (\ref{eq:negprobanal2})
\begin{equation}
\lim_{n\rightarrow\infty}P(\min_{\|\x\|_2=1}\max_{\|\lambda\|_2=1,\lambda_i\geq 0}(\g^T\lambda+\h^T\x+\kappa\lambda^T\1-\epsilon_{5}^{(g)}\sqrt{n})\geq \xi_n^{(l)})\geq 1.\label{eq:negprobanal3}
\end{equation}

We will now look at the left-hand side of the inequality in (\ref{eq:negproblemma}). The following is then the probability of interest
\begin{equation}
P(\min_{\|\x\|_2=1}\max_{\|\lambda\|_2=1,\lambda_i\geq 0}(\kappa\lambda^T\1-\lambda^TH\x+g-\epsilon_{5}^{(g)}\sqrt{n}-\xi_n^{(l)})\geq 0).\label{eq:leftnegprobanal0}
\end{equation}
Since $P(g\geq\epsilon_{5}^{(g)}\sqrt{n})<e^{-\epsilon_{6}^{(g)} n}$ (where $\epsilon_{6}^{(g)}$ is, as all other $\epsilon$'s in this paper are, independent of $n$) from (\ref{eq:leftnegprobanal0}) we have
\begin{multline}
P(\min_{\|\x\|_2=1}\max_{\|\lambda\|_2=1,\lambda_i\geq 0}(\kappa\lambda^T\1-\lambda^TH\x+g-\epsilon_{5}^{(g)}\sqrt{n}-\xi_n^{(l)})\geq 0)
\\\leq P(\min_{\|\x\|_2=1}\max_{\|\lambda\|_2=1,\lambda_i\geq 0}(\kappa\lambda^T\1-\lambda^TH\x-\xi_n^{(l)})\geq 0)+e^{-\epsilon_{6}^{(g)} n}.\label{eq:leftnegprobanal1}
\end{multline}
When $n$ is large from (\ref{eq:leftnegprobanal1}) we then have
\begin{multline}
\hspace{-.7in}\lim_{n\rightarrow \infty}P(\min_{\|\x\|_2=1}\max_{\|\lambda\|_2=1,\lambda_i\geq 0}(\kappa\sqrt{n}\lambda^T\1-\lambda^TH\x+g-\epsilon_{5}^{(g)}\sqrt{n}-\xi_n^{(l)})\geq 0)
\leq  \lim_{n\rightarrow\infty}P(\min_{\|\x\|_2=1}\max_{\|\lambda\|_2=1,\lambda_i\geq 0}(\kappa\lambda^T\1-\lambda^TH\x-\xi_n^{(l)})\geq 0)\\
 =  \lim_{n\rightarrow\infty}P(\min_{\|\x\|_2=1}\max_{\|\lambda\|_2=1,\lambda_i\geq 0}(\kappa\lambda^T\1-\lambda^TH\x)\geq \xi_n^{(l)}).\label{eq:leftnegprobanal2}
\end{multline}
Assuming that (\ref{eq:negcondxipu}) holds, then a combination of (\ref{eq:negproblemma}), (\ref{eq:negprobanal3}), and (\ref{eq:leftnegprobanal2}) gives
\begin{equation}
\hspace{-.5in}\lim_{n\rightarrow\infty}P(\min_{\|\x\|_2=1}\max_{\|\lambda\|_2=1,\lambda_i\geq 0}(\kappa\lambda^T\1-\lambda^TH\x)\geq \xi_n^{(l)})\geq \lim_{n\rightarrow\infty}P(\min_{\|\x\|_2=1}\max_{\|\lambda\|_2=1,\lambda_i\geq 0}(\g^T\y+\h^T\x+\kappa\lambda^T\1-\epsilon_{5}^{(g)}\sqrt{n})\geq \xi_n^{(l)})\geq 1.\label{eq:leftnegprobanal3}
\end{equation}

We summarize our results from this subsection in the following lemma.

\begin{lemma}
Let $H$ be an $m\times n$ matrix with i.i.d. standard normal components. Let $n$ be large and let $m=\alpha n$, where $\alpha>0$ is a constant independent of $n$. Let $\xi_n$ be as in (\ref{eq:uncorminmax}) and let $\kappa\geq 0$ be a scalar constant independent of $n$. Let all $\epsilon$'s be arbitrarily small constants independent of $n$. Further, let $\g_i$ be a standard normal random variable and set
\begin{equation}
f_{gar}(\kappa)=\frac{1}{\sqrt{2\pi}}\int_{-\kappa}^{\infty}(\g_i+\kappa)^2e^{-\frac{\g_i^2}{2}}d\g_i.\label{eq:fgarlemmaunncorlb}
\end{equation}
Let $\xi_n^{(l)}$ be a scalar such that
\begin{equation}
(1-\epsilon_{1}^{(m)})\sqrt{\alpha f_{gar}(\kappa)}-(1+\epsilon_{1}^{(n)})-\epsilon_{5}^{(g)}>\frac{\xi_n^{(l)}}{\sqrt{n}}.\label{eq:negcondxipuneggenlemma}
\end{equation}
Then
\begin{equation}
 \lim_{n\rightarrow\infty}P(\xi_n\geq \xi_n^{(l)})=\lim_{n\rightarrow\infty}P(\min_{\|\x\|_2=1}\max_{\|\lambda\|_2=1,\lambda_i\geq 0}(\kappa\lambda^T\1-\lambda^TH\x)\geq \xi_n^{(l)})\geq 1. \label{eq:neggenproblemma}
\end{equation}
\label{lemma:neggenlemma}
\end{lemma}
\begin{proof}
The proof follows from the above discussion, (\ref{eq:negproblemma}), and (\ref{eq:leftnegprobanal3}).
\end{proof}
In a more informal language (essentially ignoring all technicalities and $\epsilon$'s) one has that as long as
\begin{equation}
\alpha>\frac{1}{f_{gar}(\kappa)},\label{eq:condalphauncorlb}
\end{equation}
the problem in (\ref{eq:defprobucor1}) will be infeasible with overwhelming probability.

\subsection{Upper-bounding (the sign of) $\xi_n$}
\label{sec:uncorgardub}

In the previous subsection we designed a lower bound on $\xi_n$ which then helped us determine an upper bound on the critical storage capacity $\alpha_c$ (essentially the one given in (\ref{eq:condalphauncorlb})). In this subsection we will provide a mechanism that can be used to upper bound a quantity similar to $\xi_n$ (which will maintain the sign of $\xi_n$). Such an upper bound then can be used to obtain a lower bound on the critical storage capacity $\alpha_c$. As mentioned above, we start by looking at a quantity very similar to $\xi_n$. First, we recognize that when $\kappa>0$ one can alternatively rewrite the feasibility problem from (\ref{eq:defprobucor1}) in the following way
\begin{eqnarray}
& & H\x\geq \kappa\nonumber \\
& & \|\x\|_2\leq 1.\label{eq:defprobucor1ub}
\end{eqnarray}
For our needs in this subsection, the feasibility problem in (\ref{eq:defprobucor1ub}) can be formulated as the following optimization problem
\begin{eqnarray}
\xi_{nr}=\min_{\x} \max_{\lambda\geq 0} & &  \kappa\lambda^T\1- \lambda^T H\x \nonumber \\
\mbox{subject to} 
& & \|\lambda\|_2\leq 1\nonumber \\
& & \|\x\|_2\leq 1.\label{eq:uncorminmax}
\end{eqnarray}
For (\ref{eq:defprobucor1ub}) to be infeasible one has to have $\xi_{nr}>0$. Using duality one has
\begin{eqnarray}
\xi_{nr}= \max_{\lambda\geq 0} \min_{\x} & &  \kappa\lambda^T\1- \lambda^T H\x \nonumber \\
\mbox{subject to} 
& & \|\lambda\|_2\leq 1\nonumber \\
& & \|\x\|_2\leq 1,\label{eq:uncormaxmin}
\end{eqnarray}
and alternatively
\begin{eqnarray}
-\xi_{nr}= \min_{\lambda\geq 0} \max_{\x} & &  -\kappa\lambda^T\1+\lambda^T H\x \nonumber \\
\mbox{subject to} 
& & \|\lambda\|_2\leq 1\nonumber \\
& & \|\x\|_2\leq 1.\label{eq:uncormaxmin}
\end{eqnarray}
We will now proceed in a fashion similar to the on presented in the previous subsection. We will make use of the following lemma (the lemma is fairly similar to Lemma \ref{lemma:negproblemma} and of course fairly similar to Lemma 3.1 in \cite{Gordon88}; see also \cite{StojnicHopBnds10} for similar considerations).
\begin{lemma}
Let $H$ be an $m\times n$ matrix with i.i.d. standard normal components. Let $\g$ and $\h$ be $m\times 1$ and $n\times 1$ vectors, respectively, with i.i.d. standard normal components. Also, let $g$ be a standard normal random variable and let $\zeta_{\lambda}$ be a function of $\x$. Then
\begin{equation}
P(\min_{\|\lambda\|_2\leq 1,\lambda_i\geq 0}\max_{\|\x\|_2\leq 1}(\lambda^T H\x+g\|\lambda\|_2\|\x\|_2-\zeta_{\lambda})\geq 0)\geq
P(\min_{\|\lambda\|_2\leq 1,\lambda_i\geq 0}\max_{\|\x\|_2\leq 1}(\|\x\|_2\g^T\lambda+\|\lambda\|_2\h^T\x-\zeta_{\lambda})\geq 0).\label{eq:negproblemmaub}
\end{equation}\label{lemma:negproblemmaub}
\end{lemma}
\begin{proof}
The discussion related to the proof of Lemma \ref{lemma:negproblemma} applies here as well.
\end{proof}

Let $\zeta_{\lambda}=\kappa\lambda^T\1+\epsilon_{5}^{(g)}\sqrt{n}\|\lambda\|_2\|\x\|_2$ with $\epsilon_{5}^{(g)}>0$ being an arbitrarily small constant independent of $n$. We will follow the strategy of the previous subsection and start by first looking at the right-hand side of the inequality in (\ref{eq:negproblemmaub}). The following is then the probability of interest
\begin{equation}
P(\min_{\|\lambda\|_2\leq 1,\lambda_i\geq 0,\lambda\neq 0}\max_{\|\x\|_2\leq 1}(\|\x\|_2\g^T\lambda+\|\lambda\|_2\h^T\x-\kappa\lambda^T\1-\epsilon_{5}^{(g)}\sqrt{n}\|\lambda\|_2\|\x\|_2)>0),\label{eq:negprobanal0ub}
\end{equation}
where for the easiness of writing we removed possibility $\lambda=0$ (also, such a case contributes in no way to the possibility that $-\xi_{nr}<0)$.
After solving the minimization over $\x$ and the maximization over $\lambda$ one obtains
\begin{multline}
P(\min_{\|\lambda\|_2\leq 1,\lambda_i\geq 0,\lambda\neq 0}\max_{\|\x\|_2\leq 1}(\|\x\|_2\g^T\lambda+\|\lambda\|_2\h^T\x-\kappa\lambda^T\1-\epsilon_{5}^{(g)}\sqrt{n}\|\lambda\|_2\|\x\|_2)> 0)\\=P(\min_{\|\lambda\|_2\leq 1,\lambda_i\geq 0,\lambda\neq 0}(\max(0,\|\h\|_2\|\lambda\|_2+\g\lambda-\epsilon_{5}^{(g)}\sqrt{n}\|\lambda\|_2)-\kappa\lambda^T\1)>0).\label{eq:negprobanal1ub}
\end{multline}
Now, we will for a moment assume that $m$ and $n$ are such that
\begin{equation}
\lim_{n\rightarrow\infty}P(\min_{\|\lambda\|_2\leq 1,\lambda_i\geq 0,\lambda\neq 0}(\|\h\|_2\|\lambda\|_2+\g\lambda-\epsilon_{5}^{(g)}\sqrt{n}\|\lambda\|_2-\kappa\lambda^T\1)> 0)=1.\label{eq:negprobanal2ub}
\end{equation}
That would also imply that
\begin{equation}
\lim_{n\rightarrow\infty}P(\min_{\|\lambda\|_2\leq 1,\lambda_i\geq 0,\lambda\neq 0}(\max(0,\|\h\|_2\|\lambda\|_2+\g\lambda-\epsilon_{5}^{(g)}\sqrt{n}\|\lambda\|_2)-\kappa\lambda^T\1)>0)=1.\label{eq:negprobanal3ub}
\end{equation}
What is then left to be done is to determine an $\alpha=\frac{m}{n}$ such that (\ref{eq:negprobanal2ub}) holds. One then easily has
\begin{multline}
P(\min_{\|\lambda\|_2\leq 1,\lambda_i\geq 0,\lambda\neq 0}(\|\h\|_2\|\lambda\|_2+\g\lambda-\epsilon_{5}^{(g)}\sqrt{n}\|\lambda\|_2-\kappa\lambda^T\1)> 0)\\=
P(\min_{\|\lambda\|_2\leq 1,\lambda_i\geq 0,\lambda\neq 0}\|\lambda\|_2(\|\h\|_2-\|(\g-\kappa\1)_+\|_2-\epsilon_{5}^{(g)}\sqrt{n})> 0),\label{eq:negprobanal4ub}
\end{multline}
where similarly to what we had in the previous subsection $(\g-\kappa\1)_-$ is $(\g-\kappa\1)$ vector with positive components replaced by zeros. Since $\h$ is a vector of $n$ i.i.d. standard normal variables it is rather trivial that
\begin{equation}
P(\|\h\|_2>(1-\epsilon_{1}^{(n)})\sqrt{n})\geq 1-e^{-\epsilon_{2}^{(n)} n},\label{eq:devhub}
\end{equation}
where $\epsilon_{1}^{(n)}>0$ is an arbitrarily small constant and $\epsilon_{2}^{(n)}$ is a constant dependent on $\epsilon_{1}^{(n)}$ but independent of $n$. Along the same lines, since $\g$ is a vector of $m$ i.i.d. standard normal variables it easily follows that
\begin{equation}
E\sum_{i=1}^{n}(\min\{\g_i-\kappa,0\})^2=mf_{gar}(\kappa),\label{eq:efgarub}
\end{equation}
where we recall
\begin{equation}
f_{gar}(\kappa)=\frac{1}{\sqrt{2\pi}}\int_{-\kappa}^{\infty}(\g_i+\kappa)^2e^{-\frac{\g_i^2}{2}}d\g_i
=\frac{1}{\sqrt{2\pi}}\int_{-\infty}^{\kappa}(\g_i-\kappa)^2e^{-\frac{\g_i^2}{2}}d\g_i.\label{eq:fgar}
\end{equation}
One then easily also has
\begin{equation}
P\left (\sqrt{\sum_{i=1}^{n}(\min\{\g_i-\kappa,0\})^2}<(1+\epsilon_{1}^{(m)})\sqrt{mf_{gar}(\kappa)}\right )\geq 1-e^{-\epsilon_{2}^{(m)} m},\label{eq:devgub}
\end{equation}
where we recall that $\epsilon_{1}^{(m)}>0$ is an arbitrarily small constant and $\epsilon_{2}^{(m)}$ is a constant dependent on $\epsilon_{1}^{(m)}$ and
$f_{gar}(\kappa)$ but independent of $n$. Then a combination of (\ref{eq:negprobanal4ub}), (\ref{eq:devhub}), and (\ref{eq:devgub}) gives
\begin{multline}
P(\min_{\|\lambda\|_2\leq 1,\lambda_i\geq 0,\lambda\neq 0}(\|\h\|_2\|\lambda\|_2+\g\lambda-\epsilon_{5}^{(g)}\sqrt{n}\|\lambda\|_2-\kappa\lambda^T\1)> 0)\\\geq
(1-e^{-\epsilon_{2}^{(m)} m})(1-e^{-\epsilon_{2}^{(n)} n})
P((1-\epsilon_{1}^{(n)})\sqrt{n}-(1+\epsilon_{1}^{(m)})\sqrt{mf_{gar}(\kappa)}-\epsilon_{5}^{(g)}\sqrt{n}> 0).
\label{eq:negprobanal22ub}
\end{multline}
If
\begin{eqnarray}
& & (1-\epsilon_{1}^{(n)})\sqrt{n}-(1+\epsilon_{1}^{(m)})\sqrt{mf_{gar}(\kappa)}-\epsilon_{5}^{(g)}\sqrt{n}>0\nonumber \\
& \Leftrightarrow & (1-\epsilon_{1}^{(n)})-(1+\epsilon_{1}^{(m)})\sqrt{\alpha f_{gar}(\kappa)}-\epsilon_{5}^{(g)}>0,\label{eq:negcondxipuub}
\end{eqnarray}
one then has from (\ref{eq:negprobanal22ub})
\begin{equation}
\lim_{n\rightarrow\infty}P(\min_{\|\lambda\|_2\leq 1,\lambda_i\geq 0,\lambda\neq 0}(\|\h\|_2\|\lambda\|_2+\g\lambda-\epsilon_{5}^{(g)}\sqrt{n}\|\lambda\|_2-\kappa\lambda^T\1)> 0)\geq 1.\label{eq:negprobanal33ub}
\end{equation}
A combination of  (\ref{eq:negprobanal1ub}),  (\ref{eq:negprobanal2ub}),  (\ref{eq:negprobanal3ub}), and  (\ref{eq:negprobanal33ub}) gives that if (\ref{eq:negcondxipuub}) holds then
\begin{equation}
\lim_{n\rightarrow\infty}P(\min_{\|\lambda\|_2\leq 1,\lambda_i\geq 0,\lambda\neq 0}\max_{\|\x\|_2\leq 1}(\|\x\|_2\g^T\lambda+\|\lambda\|_2\h^T\x-\kappa\lambda^T\1-\epsilon_{5}^{(g)}\sqrt{n}\|\lambda\|_2\|\x\|_2)> 0)\geq 1.\label{eq:negprobanal44ub}
\end{equation}

We will now look at the left-hand side of the inequality in (\ref{eq:negproblemmaub}). The following is then the probability of interest
\begin{equation}
P(\min_{\|\lambda\|_2\leq 1,\lambda_i\geq 0}\max_{\|\x\|_2\leq 1}(\lambda^TH\x-\kappa\lambda^T\1+(g-\epsilon_{5}^{(g)}\sqrt{n})\|\lambda\|_2\|\x\|_2)\geq 0).\label{eq:leftnegprobanal0ub}
\end{equation}
Since $P(g\geq\epsilon_{5}^{(g)}\sqrt{n})<e^{-\epsilon_{6}^{(g)} n}$ (where $\epsilon_{6}^{(g)}$ is, as all other $\epsilon$'s in this paper are, independent of $n$) from (\ref{eq:leftnegprobanal0ub}) we have
\begin{multline}
P(\min_{\|\lambda\|_2\leq 1,\lambda_i\geq 0}\max_{\|\x\|_2\leq 1}(\lambda^TH\x-\kappa\lambda^T\1+(g-\epsilon_{5}^{(g)}\sqrt{n})\|\lambda\|_2\|\x\|_2)\geq 0)\\\leq P(\min_{\|\lambda\|_2\leq 1,\lambda_i\geq 0}\max_{\|\x\|_2\leq 1}(\lambda^TH\x-\kappa\lambda^T\1)\geq 0)+e^{-\epsilon_{6}^{(g)} n}.\label{eq:leftnegprobanal1ub}
\end{multline}
When $n$ is large from (\ref{eq:leftnegprobanal1ub}) we then have
\begin{multline}
\hspace{-.7in}\lim_{n\rightarrow \infty}P(\min_{\|\lambda\|_2\leq 1,\lambda_i\geq 0}\max_{\|\x\|_2\leq 1}(\lambda^TH\x-\kappa\lambda^T\1+(g-\epsilon_{5}^{(g)}\sqrt{n})\|\lambda\|_2\|\x\|_2)\geq 0)\\\leq \lim_{n\rightarrow \infty}P(\min_{\|\lambda\|_2\leq 1,\lambda_i\geq 0}\max_{\|\x\|_2\leq 1}(\lambda^TH\x-\kappa\lambda^T\1)\geq 0).\label{eq:leftnegprobanal2ub}
\end{multline}
Assuming that (\ref{eq:negcondxipuub}) holds, then a combination of (\ref{eq:uncormaxmin}), (\ref{eq:negproblemmaub}), (\ref{eq:negprobanal44ub}), and (\ref{eq:leftnegprobanal2ub}) gives
\begin{eqnarray}
\lim_{n\rightarrow \infty}P(\xi_{nr}\leq 0) & = & \lim_{n\rightarrow \infty}P(-\xi_{nr}\geq 0)\nonumber \\
& = & \lim_{n\rightarrow \infty}P(\min_{\|\lambda\|_2\leq 1,\lambda_i\geq 0}\max_{\|\x\|_2\leq 1}(\lambda^TH\x-\kappa\lambda^T\1)\geq 0)
\nonumber \\
& \geq &
\lim_{n\rightarrow \infty}P(\min_{\|\lambda\|_2\leq 1,\lambda_i\geq 0}\max_{\|\x\|_2\leq 1}(\lambda^TH\x-\kappa\lambda^T\1+(g-\epsilon_{5}^{(g)}\sqrt{n})\|\lambda\|_2\|\x\|_2)\geq 0)\nonumber \\
& \geq &
\lim_{n\rightarrow\infty}P(\min_{\|\lambda\|_2\leq 1,\lambda_i\geq 0,\lambda\neq 0}\max_{\|\x\|_2\leq 1}(\|\x\|_2\g^T\lambda+\|\lambda\|_2\h^T\x-\kappa\lambda^T\1-\epsilon_{5}^{(g)}\sqrt{n}\|\lambda\|_2\|\x\|_2)> 0)\nonumber \\
& \geq & 1.\label{eq:leftnegprobanal3ub}
\end{eqnarray}
From (\ref{eq:leftnegprobanal3ub}) one then has
\begin{equation}
\lim_{n\rightarrow \infty}P(\xi_{nr}> 0)=1-\lim_{n\rightarrow \infty}P(\xi_{nr}\leq 0)\leq 0,\label{eq:leftnegprobanal4ub}
\end{equation}
which implies that (\ref{eq:defprobucor1ub}) is feasible with overwhelming probability if (\ref{eq:negcondxipuub}) holds.

We summarize our results from this subsection in the following lemma.

\begin{lemma}
Let $H$ be an $m\times n$ matrix with i.i.d. standard normal components. Let $n$ be large and let $m=\alpha n$, where $\alpha>0$ is a constant independent of $n$. Let $\xi_n$ be as in (\ref{eq:uncorminmax}) and let $\kappa\geq 0$ be a scalar constant independent of $n$. Let all $\epsilon$'s be arbitrarily small constants independent of $n$. Further, let $\g_i$ be a standard normal random variable and set
\begin{equation}
f_{gar}(\kappa)=\frac{1}{\sqrt{2\pi}}\int_{-\kappa}^{\infty}(\g_i+\kappa)^2e^{-\frac{\g_i^2}{2}}d\g_i
=\frac{1}{\sqrt{2\pi}}\int_{-\infty}^{\kappa}(\g_i-\kappa)^2e^{-\frac{\g_i^2}{2}}d\g_i.\label{eq:fgarlemmaunncorub}
\end{equation}
Let $\alpha>0$ be such that
\begin{equation}
(1-\epsilon_{1}^{(n)})-(1+\epsilon_{1}^{(m)})\sqrt{\alpha f_{gar}(\kappa)}-\epsilon_{5}^{(g)}>0.\label{eq:negcondxipuneggenlemmaub}
\end{equation}
Then
\begin{equation}
\lim_{n\rightarrow \infty}P(-\xi_{nr}\geq 0)=\lim_{n\rightarrow \infty}P(-\xi_{nr}\geq 0) = \lim_{n\rightarrow\infty}P(\min_{\|\lambda\|_2\leq 1,\lambda_i\geq 0,\lambda\neq 0}\max_{\|\x\|_2\leq 1}(\kappa\lambda^T\1-\lambda^TH\x)\geq 0)\geq 1. \label{eq:neggenproblemmaub}
\end{equation}
Moreover,
\begin{equation}
\lim_{n\rightarrow \infty}P(\xi_{nr}> 0)=1-\lim_{n\rightarrow \infty}P(\xi_{nr}\leq 0)\leq 0,\label{eq:neggenproblemmaub1}
\end{equation}
\label{lemma:neggenlemmaub}
\end{lemma}
\begin{proof}
Follows from the above discussion.
\end{proof}
Similarly to what was done in the previous subsection, one can again be a bit more informal and ignore all technicalities and $\epsilon$'s. After doing so one has that as long as
\begin{equation}
\alpha<\frac{1}{f_{gar}(\kappa)},\label{eq:condalphauncorub}
\end{equation}
the problem in (\ref{eq:defprobucor1}) will be feasible with overwhelming probability. Moreover, combining results of Lemmas \ref{lemma:neggenlemma} and \ref{lemma:neggenlemmaub} one obtains (of course in an informal language) for the storage capacity $\alpha_c$
\begin{equation}
\alpha_c=\frac{1}{f_{gar}(\kappa)}.\label{eq:stcapposk}
\end{equation}
The value obtained for the storage capacity in (\ref{eq:stcapposk}) matches the one obtained in \cite{Gar88} while utilizing the replica approach. In \cite{SchTir02,SchTir03} as well as in \cite{TalBook} the above was then rigorously established as the storage capacity. In fact a bit more is shown in \cite{SchTir02,SchTir03,TalBook}. Namely, the authors considered a partition function type of quantity (i.e. a free energy type of quantity) and determined its behavior in the entire temperature regime. The storage capacity is essentially obtained based on the ground state (zero-temperature) behavior of such a free energy.

\subsection{Negative $\kappa$}
\label{sec:negkappa}

In \cite{TalBook}, Talagrand raised the question related to the behavior of the spherical perceptron when $\kappa<0$. Along the same lines, in \cite{TalBook}, Conjecture 8.4.4 was formulated where it was stated that if $\alpha>\frac{1}{f_{gar}(\kappa)}$ then the problem in (\ref{eq:defprobucor1}) is infeasible with overwhelming probability. The fact that $\kappa>0$ was never really used in our derivations in Subsection \ref{sec:uncorgardlb}. In other words, the entire derivation presented in Subsection \ref{sec:uncorgardlb} will hold even if $\kappa<0$. The results of Lemma \ref{lemma:negproblemma} then imply that for any $\kappa$ if
\begin{equation}
\alpha>\frac{1}{f_{gar}(\kappa)},\label{eq:condalphauncorlbnegkappa}
\end{equation}
then the problem in (\ref{eq:defprobucor1}) is infeasible with overwhelming probability. This resolves Talagrand's conjecture 8.4.4 from \cite{TalBook} in positive. Along the same lines, it partially answers the question (problem) 8.4.2 from \cite{TalBook} as well.

\section{Correlated Gardner problem}
\label{sec:corgard}

What we considered in the previous section is the standard Gardner problem or the standard spherical perceptron. Such a perceptron assume that all patterns (essentially rows of $H$) are uncorrelated. In \cite{Gar88} a correlated version of the problem was considered as well. The, following, relatively simple, type of correlations was analyzed: instead of assuming that al elements of $H$ are i.i.d. symmetric Bernoulli random variables, one can assume that each $H_{ij}$ is an asymmetric Bernoulli random variable. To be a bit more precise, the following type of asymmetry was considered:
\begin{eqnarray}
P(H_{ij}=1) & = & \frac{1+m_a}{2}\nonumber \\
P(H_{ij}=-1) & = & \frac{1-m_a}{2}.\label{eq:defasymgard}
\end{eqnarray}
In other words, each $H_{ij}$ was assumed to take value $1$ with probability $\frac{1+m_a}{2}$ and value $-1$ with probability $\frac{1-m_a}{2}$. Clearly, $0\leq m_a\leq 1$. If $m_a=0$ one has fully uncorrelated scenario (essentially, precisely the scenario considered in Section \ref{sec:uncorgard}. On the other hand, if $m_a=1$ one has fully correlated scenario where all patterns are basically equal to each other. Of course, one then wonders in what way the above introduced correlations impact the value of the storage capacity. The first observation is that as the correlation grow, i.e. as $m_a$ grows, one expects that the storage capacity should grow as well. Such a prediction was indeed confirmed through the analysis conducted in \cite{Gar88}. In fact, not only that such an expectations was confirmed, actually the exact changed in the storage capacity was quantified as well. In this section we will present a mathematically rigorous approach that will confirm the predictions given in \cite{Gar88}.

We start by recalling how the problems in (\ref{eq:defprobucor}) and (\ref{eq:defprobucor1}) transform when the patterns are correlated. Essentially instead of (\ref{eq:defprobucor}) one then looks at the following question: assuming that $\|\x\|_2=1$, how large $\alpha$ can be so that the following system of linear inequalities is satisfied with overwhelming probability
\begin{equation}
\mbox{diag}(H_{:,1}) H_{:,2:n}\x\geq \kappa,\label{eq:defprobcor}
\end{equation}
where $H_{:,1}$ is the first column of $H$ and $H_{:,2:n}$ are all columns of $H$ except column $1$. Also, $\mbox{diag}(H_{:,1})$ is a diagonal matrix with elements on the main diagonal being the elements of $H_{:,1}$. This of course is the same as if one asks how large $\alpha$ can be so that the following optimization problem is feasible with overwhelming probability
\begin{eqnarray}
& & \mbox{diag}(\q)H\x\geq \kappa\nonumber \\
& & \|\x\|_2=1,\label{eq:defprobcor1}
\end{eqnarray}
where elements of $\q$ and $H$ are i.i.d. asymmetric Bernoulli distributed according to (\ref{eq:defasymgard}). Also, the size of $H$ in (\ref{eq:defprobcor1}) should be $m\times (n-1)$. However, as in the previous section to make writing easier we will view it as an $m\times n$ matrix. Given that we will consider the large $n$ scenario this effectively changes nothing in the results that we will present.

Now, our strategy will be to condition on first solve the resulting problem one obtains after conditioning on $\q$. So, for the time being we will assume that $\q$ is a deterministic vector. Also, in such a scenario one can then replace the asymmetric Bernoulli random variables of $H$ by the appropriately adjusted Gaussian ones. We will not proof here that such a replacement is allowed. While we will towards the end of the paper say a few more words about it, here we just briefly mention that the proof of such a statement is not that hard since it relies on several routine techniques (see, e.g. \cite{Chatterjee06,Lindeberg22}). However, it is a bit tedious and in our opinion would significantly burden the presentation.

The adjustment to the Gaussian scenario can be done in the following way: one can assume that all components of $H$ are i.i.d. and that each of them (basically $H_{ij},1\leq i\leq m,1\leq j\leq n$) is an ${\cal N}(m_a,1-m_a^2)$. Alternatively one can assume that all components of $H$ are i.i.d. and that each of them is standard normal. Under such an assumption one then can rewrite (\ref{eq:defprobcor1}) in the following way
\begin{eqnarray}
& & \mbox{diag}(\q)(\sqrt{1-m_a^2}H+m_a)\x\geq \kappa\nonumber \\
& & \|\x\|_2=1.\label{eq:defprobcor2}
\end{eqnarray}
After further algebraic transformation one has the following feasibility problem
\begin{eqnarray}
& & \mbox{diag}(\q)H\x\geq \frac{\kappa \1-vm_a\q}{\sqrt{1-m_a^2}}\nonumber \\
& & \1^T\x=v\nonumber \\
& & \|\x\|_2=1.\label{eq:defprobcor3}
\end{eqnarray}
We should also recognize that that above feasibility problem can be rewritten as the following optimization problem
\begin{eqnarray}
\xi_{ncor}=\min_{\x} \max_{\lambda\geq 0} & &  \frac{\kappa\lambda^T\1-vm_a\lambda^T\q}{\sqrt{1-m_a^2}}- \lambda^T\mbox{diag}(\q) H\x \nonumber \\
\mbox{subject to} & & \1^T\x=v \nonumber \\
& & \|\lambda\|_2= 1\nonumber \\
& & \|\x\|_2=1.\label{eq:corminmax}
\end{eqnarray}
In what follows we will analyze the feasibility of (\ref{eq:defprobcor3}) by trying to follow closely what was presented in Subsection \ref{sec:uncorgardlb}. We will first present a mechanism that can be used to lower bound $\xi_{ncor}$ and then its a counterpart that can be used to upper bound a quantity similar to $\xi_{ncor}$ which basically has the same sign as $\xi_{ncor}$.

\subsection{Lower-bounding $\xi_{ncor}$}
\label{sec:corgardlb}

We will start with the following lemma (essentially a counterpart to Lemma \ref{lemma:neggenlemma}).
\begin{lemma}
Let $H$ be an $m\times n$ matrix with i.i.d. standard normal components. Let $\q$ be a fixed $n\times 1$ vector and let $\g$ and $\h$ be $m\times 1$ and $n\times 1$ vectors, respectively, with i.i.d. standard normal components. Also, let $g$ be a standard normal random variable and let $\zeta_{\lambda}$ be a function of $\x$. Then
\begin{equation}
\hspace{-.5in}P(\min_{\1^T\x=v,\|\x\|_2=1}\max_{\|\lambda\|_2=1,\lambda_i\geq 0}(-\lambda^T \mbox{diag}(\q)H\x+g-\zeta_{\lambda})\geq 0)\geq
P(\min_{\1^T\x=v,\|\x\|_2=1}\max_{\|\lambda\|_2=1,\lambda_i\geq 0}(\g^T\lambda \mbox{diag}(\q)+\h^T\x-\zeta_{\lambda})\geq 0).\label{eq:negproblemmacor}
\end{equation}\label{lemma:negproblemmacor}
\end{lemma}
\begin{proof}
The proof is basically similar to the proof of Lemma 3.1 in \cite{Gordon88} as well as to the proof of Lemma 7 in \cite{StojnicHopBnds10}. The only difference is the structure of the allowed sets of $\x$'s and $\lambda$'s. After recognizing that $\|\lambda^T\mbox{diag}(\q)\|_2=\|\lambda^T\|_2$, such a difference changes nothing structurally in the proof, though.
\end{proof}

Let $\zeta_{\lambda}=-\frac{\kappa\lambda^T\1-vm_a\lambda^T\q}{\sqrt{1-m_a^2}}+\epsilon_{5}^{(g)}\sqrt{n}+\xi_{ncor}^{(l)}$ with $\epsilon_{5}^{(g)}>0$ being an arbitrarily small constant independent of $n$. We will first look at the right-hand side of the inequality in (\ref{eq:negproblemmacor}). The following is then the probability of interest
\begin{equation}
P(\min_{\1^T\x=v,\|\x\|_2=1}\max_{\|\lambda\|_2=1,\lambda_i\geq 0}(\g^T\mbox{diag}(\q)\lambda+\h^T\x+\frac{\kappa\lambda^T\1-vm_a\lambda^T\q}{\sqrt{1-m_a^2}}-\epsilon_{5}^{(g)}\sqrt{n})\geq \xi_{ncor}^{(l)}).\label{eq:negprobanal0cor}
\end{equation}
After solving the minimization over $\x$ and the maximization over $\lambda$ one obtains
\begin{multline}
P(\min_{\1^T\x=v,\|\x\|_2=1}\max_{\|\lambda\|_2=1,\lambda_i\geq 0}(\g^T\mbox{diag}(\q)\lambda+\h^T\x+\frac{\kappa\lambda^T\1-vm_a\lambda^T\q}{\sqrt{1-m_a^2}}-\epsilon_{5}^{(g)}\sqrt{n})\geq \xi_{ncor}^{(l)})\\\geq P(\min_v(\|(\mbox{diag}(\q)\g+\frac{\kappa\1-vm_a\q}{\sqrt{1-m_a^2}})_+\|_2-\|\h_i\|_2-\epsilon_{5}^{(g)}\sqrt{n})\geq \xi_{ncor}^{(l)}),\label{eq:negprobanal1cor}
\end{multline}
where $(\mbox{diag}(\q)\g+\frac{\kappa\1-vm_a\q}{\sqrt{1-m_a^2}})_+$ is $(\mbox{diag}(\q)\g+\frac{\kappa\1-vm_a\q}{\sqrt{1-m_a^2}})$ vector with negative components replaced by zeros. As in Subsection \ref{sec:uncorgardlb}, since $\h$ is a vector of $n$ i.i.d. standard normal variables it is rather trivial that
\begin{equation}
P(\|\h\|_2<(1+\epsilon_{1}^{(n)})\sqrt{n})\geq 1-e^{-\epsilon_{2}^{(n)} n},\label{eq:devhcor}
\end{equation}
where $\epsilon_{1}^{(n)}>0$ is an arbitrarily small constant and $\epsilon_{2}^{(n)}$ is a constant dependent on $\epsilon_{1}^{(n)}$ but independent of $n$. Along the same lines, since $\g$ is a vector of $m$ i.i.d. standard normal random variables and $\q$ is a vector of $m$ i.i.d. asymmetric Bernouilli random variables one has
\begin{equation}
\min_v(E\sum_{i=1}^{m}(\max\{\g_i\q_i+\frac{\kappa-vm_a\q_i}{\sqrt{1-m_a^2}},0\})^2)=mf_{gar}^{(cor)}(\kappa),\label{eq:efgarcor}
\end{equation}
where the randomness is over both $\g$ and $\q$ and
\begin{multline}
f_{gar}^{(cor)}(\kappa)=\min_v( \frac{1+m_a}{2}\left ( \frac{1}{\sqrt{2\pi}}\int_{-\frac{\kappa-vm_a}{\sqrt{1-m_a^2}}}^{\infty}
\left (\g_i+\frac{\kappa-vm_a}{\sqrt{1-m_a^2}}\right )^2e^{-\frac{\g_i^2}{2}}d\g_i\right )
\\+\frac{1-m_a}{2}\left ( \frac{1}{\sqrt{2\pi}}\int_{-\infty}^{\frac{\kappa+vm_a}{\sqrt{1-m_a^2}}}
\left (-\g_i+\frac{\kappa+vm_a}{\sqrt{1-m_a^2}}\right )^2e^{-\frac{\g_i^2}{2}}d\g_i \right ) ).\label{eq:fgarcor}
\end{multline}
Since optimal $v$ will concentrate one also has
\begin{equation}
P\left (\sqrt{\min_v(\sum_{i=1}^{m}(\max\{\g_i\q_i+\frac{\kappa-vm_a\q_i}{\sqrt{1-m_a^2}},0\})^2)}>(1-\epsilon_{1}^{(m,cor)})\sqrt{mf_{gar}^{(cor)}(\kappa)}\right )\geq 1-e^{-\epsilon_{2}^{(m,cor)} m},\label{eq:devgcor}
\end{equation}
where $\epsilon_{1}^{(m,cor)}>0$ is an arbitrarily small constant and analogously as above $\epsilon_{2}^{(m,cor)}$ is a constant dependent on $\epsilon_{1}^{(m,cor)}$ and
$f_{gar}^{(cor)}(\kappa)$ but independent of $n$. Then a combination of (\ref{eq:negprobanal1cor}), (\ref{eq:devhcor}), and (\ref{eq:devgcor}) gives
\begin{multline}
P(\min_{\1^T\x=v,\|\x\|_2=1}\max_{\|\lambda\|_2=1,\lambda_i\geq 0}(\g^T\mbox{diag}(\q)\lambda+\h^T\x+\frac{\kappa\lambda^T\1-vm_a\lambda^T\q}{\sqrt{1-m_a^2}}-\epsilon_{5}^{(g)}\sqrt{n})\geq \xi_{ncor}^{(l)})\\\geq
(1-e^{-\epsilon_{2}^{(m,cor)} m})(1-e^{-\epsilon_{2}^{(n)} n})
P((1-\epsilon_{1}^{(m,cor)})\sqrt{mf_{gar}^{(cor)}(\kappa)}-(1+\epsilon_{1}^{(n)})\sqrt{n}-\epsilon_{5}^{(g)}\sqrt{n}\geq \xi_{ncor}^{(l)}).
\label{eq:negprobanal2cor}
\end{multline}
If
\begin{eqnarray}
& & (1-\epsilon_{1}^{(m,cor)})\sqrt{mf_{gar}^{(cor)}(\kappa)}-(1+\epsilon_{1}^{(n)})\sqrt{n}-\epsilon_{5}^{(g)}\sqrt{n}>\xi_{ncor}^{(l)}\nonumber \\
& \Leftrightarrow & (1-\epsilon_{1}^{(m,cor)})\sqrt{\alpha f_{gar}^{(cor)}(\kappa)}-(1+\epsilon_{1}^{(n)})-\epsilon_{5}^{(g)}>\frac{\xi_{ncor}^{(l)}}{\sqrt{n}},\label{eq:negcondxipucor}
\end{eqnarray}
one then has from (\ref{eq:negprobanal2cor})
\begin{equation}
P(\min_{\1^T\x=v,\|\x\|_2=1}\max_{\|\lambda\|_2=1,\lambda_i\geq 0}(\g^T\mbox{diag}(\q)\lambda+\h^T\x+\frac{\kappa\lambda^T\1-vm_a\lambda^T\q}{\sqrt{1-m_a^2}}-\epsilon_{5}^{(g)}\sqrt{n})\geq \xi_{ncor}^{(l)})\geq 1.\label{eq:negprobanal3cor}
\end{equation}

We will now look at the left-hand side of the inequality in (\ref{eq:negproblemma}). The following is then the probability of interest
\begin{equation}
P(\min_{\1^T\x=v,\|\x\|_2=1}\max_{\|\lambda\|_2=1,\lambda_i\geq 0}(\frac{\kappa\lambda^T\1-vm_a\lambda^T\q}{\sqrt{1-m_a^2}}-\lambda^T\mbox{diag}(\q)H\x+g-\epsilon_{5}^{(g)}\sqrt{n}-\xi_{ncor}^{(l)})\geq 0).\label{eq:leftnegprobanal0cor}
\end{equation}
As in Subsection \ref{sec:uncorgardlb}, since $P(g\geq\epsilon_{5}^{(g)}\sqrt{n})<e^{-\epsilon_{6}^{(g)} n}$ (where $\epsilon_{6}^{(g)}$ is, as all other $\epsilon$'s in this paper are, independent of $n$) from (\ref{eq:leftnegprobanal0cor}) we have
\begin{multline}
P(\min_{\1^T\x=v,\|\x\|_2=1}\max_{\|\lambda\|_2=1,\lambda_i\geq 0}(\frac{\kappa\lambda^T\1-vm_a\lambda^T\q}{\sqrt{1-m_a^2}}-\lambda^T\mbox{diag}(\q)H\x+g-\epsilon_{5}^{(g)}\sqrt{n}-\xi_{ncor}^{(l)})\geq 0)\\\leq P(\min_{\1^T\x=v,\|\x\|_2=1}\max_{\|\lambda\|_2=1,\lambda_i\geq 0}(\frac{\kappa\lambda^T\1-vm_a\lambda^T\q}{\sqrt{1-m_a^2}}-\lambda^T\mbox{diag}(\q)H\x-\xi_{ncor}^{(l)})\geq 0)+e^{-\epsilon_{6}^{(g)} n}.
\label{eq:leftnegprobanal1cor}
\end{multline}
When $n$ is large from (\ref{eq:leftnegprobanal1cor}) we then have
\begin{multline}
\hspace{-.7in}\lim_{n\rightarrow \infty}P(\min_{\1^T\x=v,\|\x\|_2=1}\max_{\|\lambda\|_2=1,\lambda_i\geq 0}(\frac{\kappa\lambda^T\1-vm_a\lambda^T\q}{\sqrt{1-m_a^2}}-\lambda^T\mbox{diag}(\q)H\x+g-\epsilon_{5}^{(g)}\sqrt{n}-\xi_{ncor}^{(l)})\geq 0)\\
\leq  \lim_{n\rightarrow\infty}P(\min_{\1^T\x=v,\|\x\|_2=1}\max_{\|\lambda\|_2=1,\lambda_i\geq 0}(\frac{\kappa\lambda^T\1-vm_a\lambda^T\q}{\sqrt{1-m_a^2}}-\lambda^T\mbox{diag}(\q)H\x)\geq \xi_{ncor}^{(l)}).\label{eq:leftnegprobanal2cor}
\end{multline}
Assuming that (\ref{eq:negcondxipucor}) holds, then a combination of (\ref{eq:negproblemmacor}), (\ref{eq:negprobanal3cor}), and (\ref{eq:leftnegprobanal2cor}) gives
\begin{multline}
\lim_{n\rightarrow\infty}\lim_{n\rightarrow\infty}P(\min_{\1^T\x=v,\|\x\|_2=1}\max_{\|\lambda\|_2=1,\lambda_i\geq 0}(\frac{\kappa\lambda^T\1-vm_a\lambda^T\q}{\sqrt{1-m_a^2}}-\lambda^T\mbox{diag}(\q)H\x)\geq \xi_{ncor}^{(l)})\\\geq P(\min_{\1^T\x=v,\|\x\|_2=1}\max_{\|\lambda\|_2=1,\lambda_i\geq 0}(\g^T\mbox{diag}(\q)\lambda+\h^T\x+\frac{\kappa\lambda^T\1-vm_a\lambda^T\q}{\sqrt{1-m_a^2}}-\epsilon_{5}^{(g)}\sqrt{n})\geq \xi_{ncor}^{(l)})\geq 1.\label{eq:leftnegprobanal3cor}
\end{multline}

We summarize our results from this subsection in the following lemma.

\begin{lemma}
Let $H$ be an $m\times n$ matrix with i.i.d. standard normal components. Further, let $m_a$ be a constant such that $m_a\in[0,1]$ and let $\q$ be an $m\times 1$ vector of i.i.d. asymmetric Bernoulli random variables defined in the following way:
\begin{eqnarray}
P(\q_i=1) & = & \frac{1+m_a}{2}\nonumber \\
P(\q_i=-1) & = & \frac{1-m_a}{2}.\label{eq:defqlemmacorlb}
\end{eqnarray}
Let $n$ be large and let $m=\alpha n$, where $\alpha>0$ is a constant independent of $n$. Let $\xi_{ncor}$ be as in (\ref{eq:corminmax}) and let $\kappa\geq 0$ be a scalar constant independent of $n$. Let all $\epsilon$'s be arbitrarily small constants independent of $n$. Further, let $\g_i$ be a standard normal random variable and set
\begin{multline}
f_{gar}^{(cor)}(\kappa)=\min_v( \frac{1+m_a}{2}\left ( \frac{1}{\sqrt{2\pi}}\int_{-\frac{\kappa-vm_a}{\sqrt{1-m_a^2}}}^{\infty}
\left (\g_i+\frac{\kappa-vm_a}{\sqrt{1-m_a^2}}\right )^2e^{-\frac{\g_i^2}{2}}d\g_i\right )
\\+\frac{1-m_a}{2}\left ( \frac{1}{\sqrt{2\pi}}\int_{-\infty}^{\frac{\kappa+vm_a}{\sqrt{1-m_a^2}}}
\left (-\g_i+\frac{\kappa+vm_a}{\sqrt{1-m_a^2}}\right )^2e^{-\frac{\g_i^2}{2}}d\g_i \right ) ).
\label{eq:fgarlemmacorlb}
\end{multline}
Let $\xi_{ncor}^{(l)}$ be a scalar such that
\begin{equation}
(1-\epsilon_{1}^{(m,cor)})\sqrt{\alpha f_{gar}^{(cor)}(\kappa)}-(1+\epsilon_{1}^{(n)})-\epsilon_{5}^{(g)}>\frac{\xi_{ncor}^{(l)}}{\sqrt{n}}.\label{eq:negcondxipuneggenlemmacor}
\end{equation}
Then
\begin{equation}
 \lim_{n\rightarrow\infty}P(\xi_{ncor}\geq \xi_{ncor}^{(l)})=\lim_{n\rightarrow\infty}P(\min_{\1^T\x=v,\|\x\|_2=1}\max_{\|\lambda\|_2=1,\lambda_i\geq 0}(\frac{\kappa\lambda^T\1-vm_a\lambda^T\q}{\sqrt{1-m_a^2}}-\lambda^T \mbox{diag}(\q)H\x)\geq \xi_{ncor}^{(l)})\geq 1. \label{eq:neggenproblemmacor}
\end{equation}
\label{lemma:neggenlemmalbcor}
\end{lemma}
\begin{proof}
The proof follows from the above discussion, (\ref{eq:negproblemmacor}), and (\ref{eq:leftnegprobanal3cor}).
\end{proof}
One again can be a bit more informal and (essentially ignoring all technicalities and $\epsilon$'s) have that as long as
\begin{equation}
\alpha>\frac{1}{f_{gar}^{(cor)}(\kappa)},\label{eq:condalphacorlb}
\end{equation}
the problem in (\ref{eq:defprobcor1}) will be infeasible with overwhelming probability. Also, the above lemma establishes in a mathematically rigorous way that the critical storage capacity is indeed upper bounded as predicted in \cite{Gar88}.

\subsection{Upper-bounding (the sign of) $\xi_{ncor}$}
\label{sec:corgardub}

In the previous subsection we designed a lower bound on $\xi_{ncor}$ which then helped us determine an upper bound on the critical storage capacity $\alpha_c$ (essentially the one given in (\ref{eq:condalphacorlb})). Similarly to what was done in Subsection \ref{sec:uncorgardub} where we presented a mechanism to upper bound a quantity similar to $\xi_n$,  in this subsection we will provide a mechanism that can be used to upper bound a quantity similar to $\xi_{ncor}$ (which will maintain the sign of $\xi_{ncor}$). Such an upper bound then can be used to obtain a lower bound on the critical storage capacity $\alpha_c$ when the patterns are correlated. As mentioned above, we start by looking at a quantity very similar to $\xi_{ncor}$. First, we recognize that when $\kappa>vm_a$ one can alternatively rewrite the feasibility problem from (\ref{eq:defprobcor1}) in the following way
\begin{eqnarray}
& & \mbox{diag}(\q)H\x\geq \frac{\kappa\1-vm_a\q}{\sqrt{1-m_a^2}}\nonumber \\
& & \1^T\x=v\nonumber \\
& & \|\x\|_2\leq 1.\label{eq:defprobcor1ub}
\end{eqnarray}
For our needs in this subsection, the feasibility problem in (\ref{eq:defprobcor1ub}) can be formulated as the following optimization problem
\begin{eqnarray}
\xi_{nrcor}=\min_{\x} \max_{\lambda\geq 0} & &  \frac{\kappa\lambda^T\1-vm_a\lambda^T\q}{\sqrt{1-m_a^2}}- \lambda^T \mbox{diag}(\q) H\x \nonumber \\
\mbox{subject to} & & \1^T\x=v\nonumber \\
& & \|\lambda\|_2\leq 1\nonumber \\
& & \|\x\|_2\leq 1.\label{eq:corminmax}
\end{eqnarray}
For (\ref{eq:defprobcor1ub}) to be infeasible one has to have $\xi_{nr}>0$. Using duality one has
\begin{eqnarray}
\xi_{nrcor}=\max_{\lambda\geq 0} \min_{\x}  & &  \frac{\kappa\lambda^T\1-vm_a\lambda^T\q}{\sqrt{1-m_a^2}}- \lambda^T \mbox{diag}(\q) H\x \nonumber \\
\mbox{subject to} & & \1^T\x=v\nonumber \\
& & \|\lambda\|_2\leq 1\nonumber \\
& & \|\x\|_2\leq 1.\label{eq:cormaxmin}
\end{eqnarray}
and alternatively
\begin{eqnarray}
-\xi_{nrcor}=\min_{\lambda\geq 0} \max_{\x}  & &  -\frac{\kappa\lambda^T\1-vm_a\lambda^T\q}{\sqrt{1-m_a^2}}+ \lambda^T \mbox{diag}(\q) H\x \nonumber \\
\mbox{subject to} & & \1^T\x=v\nonumber \\
& & \|\lambda\|_2\leq 1\nonumber \\
& & \|\x\|_2\leq 1.\label{eq:cormaxminneg}
\end{eqnarray}
We will now proceed in a fashion similar to the on presented in Subsections \ref{sec:uncorgardub} and \ref{sec:corgardlb}. We will make use of the following lemma (the lemma is fairly similar to Lemma \ref{lemma:negproblemmacor} and of course fairly similar to Lemma 3.1 in \cite{Gordon88}; see also \cite{StojnicHopBnds10} for similar considerations).
\begin{lemma}
Let $H$ be an $m\times n$ matrix with i.i.d. standard normal components. Let $\g$ and $\h$ be $m\times 1$ and $n\times 1$ vectors, respectively, with i.i.d. standard normal components. Also, let $g$ be a standard normal random variable and let $\zeta_{\lambda}$ be a function of $\x$. Then
\begin{multline}
P(\min_{\|\lambda\|_2\leq 1,\lambda_i\geq 0}\max_{\1^T\x=v,\|\x\|_2\leq 1}(\lambda^T\mbox{diag}(\q) H\x+g\|\lambda\|_2\|\x\|_2-\zeta_{\lambda})\geq 0)\\\geq
P(\min_{\|\lambda\|_2\leq 1,\lambda_i\geq 0}\max_{\1^T\x=v,\|\x\|_2\leq 1}(\|\x\|_2\g^T\mbox{diag}(\q)\lambda+\|\lambda\|_2\h^T\x-\zeta_{\lambda})\geq 0).\label{eq:negproblemmaubcor}
\end{multline}\label{lemma:negproblemmaubcor}
\end{lemma}
\begin{proof}
The discussion related to the proof of Lemma \ref{lemma:negproblemma} applies here as well.
\end{proof}

Let $\zeta_{\lambda}=\frac{\kappa\lambda^T\1-vm_a\lambda^T\q}{\sqrt{1-m_a^2}}+\epsilon_{5}^{(g)}\sqrt{n}\|\lambda\|_2\|\x\|_2$ with $\epsilon_{5}^{(g)}>0$ being an arbitrarily small constant independent of $n$. We will follow the strategy of the previous subsection and start by first looking at the right-hand side of the inequality in (\ref{eq:negproblemmaubcor}). The following is then the probability of interest
\begin{equation}
P(\min_{\|\lambda\|_2\leq 1,\lambda_i\geq 0,\lambda\neq 0}\max_{\1^T\x=v,\|\x\|_2\leq 1}(\|\x\|_2\g^T\mbox{diag}(\q)\lambda+\|\lambda\|_2\h^T\x-\frac{\kappa\lambda^T\1-vm_a\lambda^T\q}{\sqrt{1-m_a^2}}-\epsilon_{5}^{(g)}
\sqrt{n}\|\lambda\|_2\|\x\|_2)>0),\label{eq:negprobanal0ubcor}
\end{equation}
where, similarly to what was done in Subsection \ref{sec:uncorgardub}, for the easiness of writing we removed possibility $\lambda=0$ (also, as earlier, such a case contributes in no way to the possibility that $-\xi_{nrcor}<0)$. After solving the minimization over $\x$ and the maximization over $\lambda$ one obtains
\begin{multline}
P(\min_{\|\lambda\|_2\leq 1,\lambda_i\geq 0,\lambda\neq 0}\max_{\1^T\x=v,\|\x\|_2\leq 1}(\|\x\|_2\g^T\lambda+\|\lambda\|_2\h^T\x-\frac{\kappa\lambda^T\1-vm_a\lambda^T\q}{\sqrt{1-m_a^2}}-\epsilon_{5}^{(g)}\sqrt{n}\|\lambda\|_2\|\x\|_2)> 0)\\=P(\min_{\|\lambda\|_2\leq 1,\lambda_i\geq 0,\lambda\neq 0}(\max(0,f(\h,v)\|\lambda\|_2+\g^T\mbox{diag}(\q)\lambda-\epsilon_{5}^{(g)}\sqrt{n}\|\lambda\|_2)-\frac{\kappa\lambda^T\1-vm_a\lambda^T\q}{\sqrt{1-m_a^2}})>0),
\label{eq:negprobanal1ubcor}
\end{multline}
where
\begin{equation}
f(\h,v)=\min_{\1^T\x=v,\|\x\|_2\leq 1}\h^T\x.\label{eq:deffhvcor}
\end{equation}
Now, we will for a moment assume that $m$ and $n$ are such that
\begin{equation}
\lim_{n\rightarrow\infty}P(\min_{\|\lambda\|_2\leq 1,\lambda_i\geq 0,\lambda\neq 0}(f(\h,v)\|\lambda\|_2+\g^T\mbox{diag}(\q)\lambda-\epsilon_{5}^{(g)}\sqrt{n}\|\lambda\|_2-\frac{\kappa\lambda^T\1-vm_a\lambda^T\q}{\sqrt{1-m_a^2}})> 0)=1.\label{eq:negprobanal2ubcor}
\end{equation}
That would also imply that
\begin{equation}
\lim_{n\rightarrow\infty}P(\min_{\|\lambda\|_2\leq 1,\lambda_i\geq 0,\lambda\neq 0}(\max(0,f(\h,v)\|\lambda\|_2+\g^T\mbox{diag}(\q)\lambda-\epsilon_{5}^{(g)}\sqrt{n}\|\lambda\|_2)-\frac{\kappa\lambda^T\1-vm_a\lambda^T\q}{\sqrt{1-m_a^2}})>0)
=1.\label{eq:negprobanal3ubcor}
\end{equation}
What is then left to be done is to determine an $\alpha=\frac{m}{n}$ such that (\ref{eq:negprobanal2ubcor}) holds. One then easily has
\begin{multline}
P(\min_{\|\lambda\|_2\leq 1,\lambda_i\geq 0,\lambda\neq 0}(f(\h,v)\|\lambda\|_2+\g^T \mbox{diag}(\q)\lambda-\epsilon_{5}^{(g)}\sqrt{n}\|\lambda\|_2-\frac{\kappa\lambda^T\1-vm_a\lambda^T\q}{\sqrt{1-m_a^2}})> 0)\\=
P(\min_{\|\lambda\|_2\leq 1,\lambda_i\geq 0,\lambda\neq 0}\|\lambda\|_2(f(\h,v)-\|(\mbox{diag}(\q)\g-\frac{\kappa\lambda^T\1-vm_a\lambda^T\q}{\sqrt{1-m_a^2}})_-\|_2-\epsilon_{5}^{(g)}\sqrt{n})> 0),\label{eq:negprobanal4ubcor}
\end{multline}
where similarly to what we had in the previous subsection $(\mbox{diag}(\q)\g-\frac{\kappa\lambda^T\1-vm_a\lambda^T\q}{\sqrt{1-m_a^2}})_-$ is $(\mbox{diag}(\q)\g-\frac{\kappa\lambda^T\1-vm_a\lambda^T\q}{\sqrt{1-m_a^2}})$ vector with positive components replaced by zeros. Similarly to what we had in the previous subsection, since $\g$ is a vector of $m$ i.i.d. standard normal variables it easily follows that
\begin{equation}
\min_v(E\sum_{i=1}^{m}(\min\{\g_i\q_i-\frac{\kappa\lambda^T\1-vm_a\lambda^T\q}{\sqrt{1-m_a^2}},0\})^2)=mf_{gar}^{(cor)}(\kappa),\label{eq:efgarubcor}
\end{equation}
where we recall
\begin{multline}
f_{gar}^{(cor)}(\kappa)=\min_v( \frac{1+m_a}{2}\left ( \frac{1}{\sqrt{2\pi}}\int_{-\infty}^{\frac{\kappa-vm_a}{\sqrt{1-m_a^2}}}
\left (\g_i-\frac{\kappa-vm_a}{\sqrt{1-m_a^2}}\right )^2e^{-\frac{\g_i^2}{2}}d\g_i\right )
\\+\frac{1-m_a}{2}\left ( \frac{1}{\sqrt{2\pi}}\int_{-\frac{\kappa+vm_a}{\sqrt{1-m_a^2}}}^{\infty}
\left (-\g_i-\frac{\kappa+vm_a}{\sqrt{1-m_a^2}}\right )^2e^{-\frac{\g_i^2}{2}}d\g_i \right ) ).\label{eq:fgarcor}
\end{multline}
As earlier, since optimal $v$ will concentrate one also has
\begin{equation}
P\left (\sqrt{\min_v(\sum_{i=1}^{m}(\min\{\g_i\q_i-\frac{\kappa-vm_a\q_i}{\sqrt{1-m_a^2}},0\})^2)}<(1+\epsilon_{1}^{(m,cor)})\sqrt{mf_{gar}^{(cor)}(\kappa)}\right )\geq 1-e^{-\epsilon_{2}^{(m,cor)} m},\label{eq:devgubcor}
\end{equation}
where $\epsilon_{1}^{(m,cor)}>0$ is an arbitrarily small constant and analogously as above $\epsilon_{2}^{(m,cor)}$ is a constant dependent on $\epsilon_{1}^{(m,cor)}$ and
$f_{gar}^{(cor)}(\kappa)$ but independent of $n$.

Now, we will look at $f(\h,v)$. Assuming that $v$ is a constant independent of $n$, from (\ref{eq:deffhvcor}) we have
\begin{equation}
f(\h,v)=\min_{\1^T\x=v,\|\x\|_2\leq 1}\h^T\x=\min_{\gamma_1}\|\h+\gamma_1\1\|_2-\gamma_1v.\label{eq:deffhvcor1}
\end{equation}
After solving over $\gamma_1$ we further have
\begin{equation}
E\gamma_1\approx\frac{v}{\sqrt{n-v^2}}.\label{eq:deffhvcor2}
\end{equation}
Moreover, $\gamma_1$ concentrates around $E\gamma_1$ with overwhelming probability and one than from (\ref{eq:deffhvcor1}) has that
\begin{equation}
\lim_{n\rightarrow \infty}\frac{Ef(\h,v)}{\sqrt{n}}=1,\label{eq:deffhvcor3}
\end{equation}
and $f(\h,v)$ concentrates around its mean with overwhelming probability, i.e.
\begin{equation}
P(f(\h,v)>(1-\epsilon_{1}^{(n,cor)})\sqrt{n})\geq 1-e^{-\epsilon_{2}^{(n,cor)} n},\label{eq:devhubcor}
\end{equation}
where $\epsilon_{1}^{(n,cor)}>0$ is an arbitrarily small constant and $\epsilon_{2}^{(n,cor)}$ is a constant dependent on $\epsilon_{1}^{(n,cor)}$ and $v$ but independent of $n$.

Then a combination of (\ref{eq:negprobanal4ubcor}), (\ref{eq:devgubcor}), and (\ref{eq:devhubcor}) gives
\begin{multline}
P(\min_{\|\lambda\|_2\leq 1,\lambda_i\geq 0,\lambda\neq 0}(f(\h,v)\|\lambda\|_2+\g^T\mbox{diag}(\q)\lambda-\epsilon_{5}^{(g)}\sqrt{n}\|\lambda\|_2-\kappa\lambda^T\1)> 0)\\\geq
(1-e^{-\epsilon_{2}^{(m,cor)} m})(1-e^{-\epsilon_{2}^{(n,cor)} n})
P((1-\epsilon_{1}^{(n,cor)})\sqrt{n}-(1+\epsilon_{1}^{(m,cor)})\sqrt{mf_{gar}^{(cor)}(\kappa)}-\epsilon_{5}^{(g)}\sqrt{n}> 0).
\label{eq:negprobanal22ubcor}
\end{multline}
If
\begin{eqnarray}
& & (1-\epsilon_{1}^{(n,cor)})\sqrt{n}- (1+\epsilon_{1}^{(m,cor)})\sqrt{mf_{gar}^{(cor)}(\kappa)}-\epsilon_{5}^{(g)}\sqrt{n}>0\nonumber \\
& \Leftrightarrow & (1-\epsilon_{1}^{(n,cor)})-(1+\epsilon_{1}^{(m,cor)})\sqrt{\alpha f_{gar}^{(cor)}(\kappa)}-\epsilon_{5}^{(g)}>0,\label{eq:negcondxipuubcor}
\end{eqnarray}
one then has from (\ref{eq:negprobanal22ubcor})
\begin{equation}
\lim_{n\rightarrow\infty}P(\min_{\|\lambda\|_2\leq 1,\lambda_i\geq 0,\lambda\neq 0}(f(\h,v)\|\lambda\|_2+\g^T\mbox{diag}(\q)\lambda-\epsilon_{5}^{(g)}\sqrt{n}\|\lambda\|_2-\frac{\kappa\lambda^T\1-vm_a\lambda^T\q}{\sqrt{1-m_a^2}})> 0)\geq 1.\label{eq:negprobanal33ubcor}
\end{equation}
A combination of  (\ref{eq:negprobanal1ubcor}),  (\ref{eq:negprobanal2ubcor}),  (\ref{eq:negprobanal3ubcor}), and  (\ref{eq:negprobanal33ubcor}) gives that if (\ref{eq:negcondxipuubcor}) holds then
\begin{equation}
\hspace{-.5in}\lim_{n\rightarrow\infty}P(\min_{\|\lambda\|_2\leq 1,\lambda_i\geq 0,\lambda\neq 0}\max_{\1^T\x=v,\|\x\|_2\leq 1}(\|\x\|_2\g^T\mbox{diag}(\q)\lambda+\|\lambda\|_2\h^T\x-\frac{\kappa\lambda^T\1-vm_a\lambda^T\q}{\sqrt{1-m_a^2}}-\epsilon_{5}^{(g)}\sqrt{n}\|\lambda\|_2\|\x\|_2)> 0)\geq 1.\label{eq:negprobanal44ubcor}
\end{equation}

We will now look at the left-hand side of the inequality in (\ref{eq:negproblemmaubcor}). The following is then the probability of interest
\begin{equation}
P(\min_{\|\lambda\|_2\leq 1,\lambda_i\geq 0}\max_{\1^T\x=v,\|\x\|_2\leq 1}(\lambda^T \mbox{diag}(\q) H\x-\frac{\kappa\lambda^T\1-vm_a\lambda^T\q}{\sqrt{1-m_a^2}}+(g-\epsilon_{5}^{(g)}\sqrt{n})\|\lambda\|_2\|\x\|_2)\geq 0).\label{eq:leftnegprobanal0ubcor}
\end{equation}
Since $P(g\geq\epsilon_{5}^{(g)}\sqrt{n})<e^{-\epsilon_{6}^{(g)} n}$ (where $\epsilon_{6}^{(g)}$ is, as all other $\epsilon$'s in this paper are, independent of $n$) from (\ref{eq:leftnegprobanal0ubcor}) we have
\begin{multline}
P(\min_{\|\lambda\|_2\leq 1,\lambda_i\geq 0}\max_{\1^T\x=v,\|\x\|_2\leq 1}(\lambda^T\mbox{diag}(\q)H\x-\frac{\kappa\lambda^T\1-vm_a\lambda^T\q}{\sqrt{1-m_a^2}}+(g-\epsilon_{5}^{(g)}\sqrt{n})\|\lambda\|_2\|\x\|_2)\geq 0)\\\leq P(\min_{\|\lambda\|_2\leq 1,\lambda_i\geq 0}\max_{\1^T\x=v,\|\x\|_2\leq 1}(\lambda^T \mbox{diag}(\q) H\x-\frac{\kappa\lambda^T\1-vm_a\lambda^T\q}{\sqrt{1-m_a^2}})\geq 0)+e^{-\epsilon_{6}^{(g)} n}.\label{eq:leftnegprobanal1ubcor}
\end{multline}
When $n$ is large from (\ref{eq:leftnegprobanal1ubcor}) we then have
\begin{multline}
\hspace{-.7in}\lim_{n\rightarrow \infty}P(\min_{\|\lambda\|_2\leq 1,\lambda_i\geq 0}\max_{\1^T\x=v,\|\x\|_2\leq 1}(\lambda^T \mbox{diag}(\q) H\x-\frac{\kappa\lambda^T\1-vm_a\lambda^T\q}{\sqrt{1-m_a^2}}+(g-\epsilon_{5}^{(g)}\sqrt{n})\|\lambda\|_2\|\x\|_2)\geq 0)\\\leq \lim_{n\rightarrow \infty}P(\min_{\|\lambda\|_2\leq 1,\lambda_i\geq 0}\max_{\1^T\x=v,\|\x\|_2\leq 1}(\lambda^T \mbox{diag}(\q)H\x-\frac{\kappa\lambda^T\1-vm_a\lambda^T\q}{\sqrt{1-m_a^2}})\geq 0).\label{eq:leftnegprobanal2ubcor}
\end{multline}
Assuming that (\ref{eq:negcondxipuubcor}) holds, then a combination of (\ref{eq:cormaxmin}), (\ref{eq:negproblemmaubcor}), (\ref{eq:negprobanal44ubcor}), and (\ref{eq:leftnegprobanal2ubcor}) gives
\begin{multline}
\lim_{n\rightarrow \infty}P(\xi_{nrcor}\leq 0)  =  \lim_{n\rightarrow \infty}P(-\xi_{nrcor}\geq 0)\\
 =  \lim_{n\rightarrow \infty}P(\min_{\|\lambda\|_2\leq 1,\lambda_i\geq 0}\max_{\1^T\x=v,\|\x\|_2\leq 1}(\lambda^T\mbox{diag}(\q)H\x-\frac{\kappa\lambda^T\1-vm_a\lambda^T\q}{\sqrt{1-m_a^2}})\geq 0)
\\
\geq
\lim_{n\rightarrow \infty}P(\min_{\|\lambda\|_2\leq 1,\lambda_i\geq 0}\max_{\1^T\x=v,\|\x\|_2\leq 1}(\lambda^T\mbox{diag}(\q)H\x-\frac{\kappa\lambda^T\1-vm_a\lambda^T\q}{\sqrt{1-m_a^2}}+(g-\epsilon_{5}^{(g)}\sqrt{n})\|\lambda\|_2\|\x\|_2)\geq 0)\\
\hspace{-.5in} \geq
\lim_{n\rightarrow\infty}P(\min_{\|\lambda\|_2\leq 1,\lambda_i\geq 0,\lambda\neq 0}\max_{\1^T\x=v,\|\x\|_2\leq 1}(\|\x\|_2\g^T\mbox{diag}(\q)\lambda+\|\lambda\|_2\h^T\x-\frac{\kappa\lambda^T\1-vm_a\lambda^T\q}{\sqrt{1-m_a^2}}-\epsilon_{5}^{(g)}\sqrt{n}\|\lambda\|_2\|\x\|_2)> 0)
\geq  1.\\\label{eq:leftnegprobanal3ubcor}
\end{multline}
From (\ref{eq:leftnegprobanal3ubcor}) one then has
\begin{equation}
\lim_{n\rightarrow \infty}P(\xi_{nrcor}> 0)=1-\lim_{n\rightarrow \infty}P(\xi_{nrcor}\leq 0)\leq 0,\label{eq:leftnegprobanal4ubcor}
\end{equation}
which implies that (\ref{eq:defprobcor1ub}) is feasible with overwhelming probability if (\ref{eq:negcondxipuubcor}) holds.

We summarize our results from this subsection in the following lemma.

\begin{lemma}
Let $H$ be an $m\times n$ matrix with i.i.d. standard normal components. Further, let $m_a$ be a constant such that $m_a\in[0,1]$ and let $\q$ be an $m\times 1$ vector of i.i.d. asymmetric Bernoulli random variables defined in the following way:
\begin{eqnarray}
P(\q_i=1) & = & \frac{1+m_a}{2}\nonumber \\
P(\q_i=-1) & = & \frac{1-m_a}{2}.\label{eq:defqlemmacorub}
\end{eqnarray}
Let $n$ be large and let $m=\alpha n$, where $\alpha>0$ is a constant independent of $n$. Let $\xi_{nrcor}$ be as in (\ref{eq:cormaxmin}). Let all $\epsilon$'s be arbitrarily small constants independent of $n$. Further, let $\g_i$ be a standard normal random variable and set
\begin{multline}
f_{gar}^{(cor)}(\kappa)=\min_v( \frac{1+m_a}{2}\left ( \frac{1}{\sqrt{2\pi}}\int_{-\infty}^{\frac{\kappa-vm_a}{\sqrt{1-m_a^2}}}
\left (\g_i-\frac{\kappa-vm_a}{\sqrt{1-m_a^2}}\right )^2e^{-\frac{\g_i^2}{2}}d\g_i\right )
\\+\frac{1-m_a}{2}\left ( \frac{1}{\sqrt{2\pi}}\int_{-\frac{\kappa+vm_a}{\sqrt{1-m_a^2}}}^{\infty}
\left (-\g_i-\frac{\kappa+vm_a}{\sqrt{1-m_a^2}}\right )^2e^{-\frac{\g_i^2}{2}}d\g_i \right ) ).
\label{eq:fgarlemmacorub}
\end{multline}
Further, let $v_{opt}$
\begin{multline}
v_{opt}=\mbox{argmin}_v( \frac{1+m_a}{2}\left ( \frac{1}{\sqrt{2\pi}}\int_{-\infty}^{\frac{\kappa-vm_a}{\sqrt{1-m_a^2}}}
\left (\g_i-\frac{\kappa-vm_a}{\sqrt{1-m_a^2}}\right )^2e^{-\frac{\g_i^2}{2}}d\g_i\right )
\\+\frac{1-m_a}{2}\left ( \frac{1}{\sqrt{2\pi}}\int_{-\frac{\kappa+vm_a}{\sqrt{1-m_a^2}}}^{\infty}
\left (-\g_i-\frac{\kappa+vm_a}{\sqrt{1-m_a^2}}\right )^2e^{-\frac{\g_i^2}{2}}d\g_i \right ) ).
\label{eq:voptlemmacorub}
\end{multline}
Then, if $\kappa\geq v_{opt}m_a$
\begin{equation}
 \lim_{n\rightarrow\infty}P(-\xi_{nrcor}\geq 0)=\lim_{n\rightarrow\infty}P(\min_{\1^T\x=v,\|\x\|_2=1}\max_{\|\lambda\|_2=1,\lambda_i\geq 0}(\frac{\kappa\lambda^T\1-vm_a\lambda^T\q}{\sqrt{1-m_a^2}}-\lambda^T \mbox{diag}(\q)H\x)\geq 0)\geq 1. \label{eq:neggenproblemmacorub}
\end{equation}
\label{lemma:neggenlemmaubcor}
\end{lemma}
\begin{proof}
The proof follows from the above discussion, (\ref{eq:negproblemmaubcor}), and (\ref{eq:leftnegprobanal3ubcor}).
\end{proof}
One again can be a bit more informal and (essentially ignoring all technicalities and $\epsilon$'s) have that as long as
\begin{equation}
\alpha<\frac{1}{f_{gar}^{(cor)}(\kappa)},\label{eq:condalphacorub}
\end{equation}
the problem in (\ref{eq:defprobcor1}) will be feasible with overwhelming probability. Also, the above lemma establishes in a mathematically rigorous way that the critical storage capacity is indeed upper bounded as predicted in \cite{Gar88}. Moreover, if $\kappa$ is as specified in Lemma \ref{lemma:negproblemmaubcor} then combining results of Lemmas \ref{lemma:neggenlemmalbcor} and \ref{lemma:neggenlemmaubcor} one obtains (of course in an informal language) for the storage capacity $\alpha_c$
\begin{equation}
\alpha_c=\frac{1}{f_{gar}^{(cor)}(\kappa)}.\label{eq:stcapposkcor}
\end{equation}
The value obtained for the storage capacity in (\ref{eq:stcapposkcor}) matches the one obtained in \cite{Gar88} while utilizing the replica approach. Below in Figures \ref{fig:alfackappa} and \ref{fig:kappakappaadj} we show how the storage capacity changes as a function of $\kappa$. More specifically, in Figure \ref{fig:alfackappa} we show how $\alpha_c$ changes as a function of $\kappa$ for three different values of correlating parameter $m_a$. Namely, we look at a the uncorrelated case $m_a=0$ and two correlated cases, $m_a=0.5$ and $m_a=0.8$. In Figure \ref{fig:kappakappaadj} we present how $\kappa_{adj}=\frac{\kappa-v_{opt}m_a}{\sqrt{1-m_a^2}}$ changes as a function of $\kappa$. Using the condition $\kappa_{adj}=0$ (which is essentially the same as $\kappa=v_{opt}m_a$) we obtain a critical value for $\kappa$, $\kappa^{(c)}$, so that the lower bound given in Lemma \ref{lemma:neggenlemmaubcor} holds (as in Section \ref{sec:uncorgard}, the upper bound given in Lemma \ref{lemma:neggenlemmalbcor} holds for any $\kappa$). These critical values are shown in Figure \ref{fig:kappakappaadj} as well together with the corresponding values for the storage capacity. These values are also presented in Figure \ref{fig:alfackappa}, which then essentially establishes the curves given in Figure \ref{fig:alfackappa} as the exact storage capacity values in the regimes to the left of the vertical bars and as rigorous upper bounds in the regime to the right of the vertical bars.

\begin{figure}[htb]
\centering
\centerline{\epsfig{figure=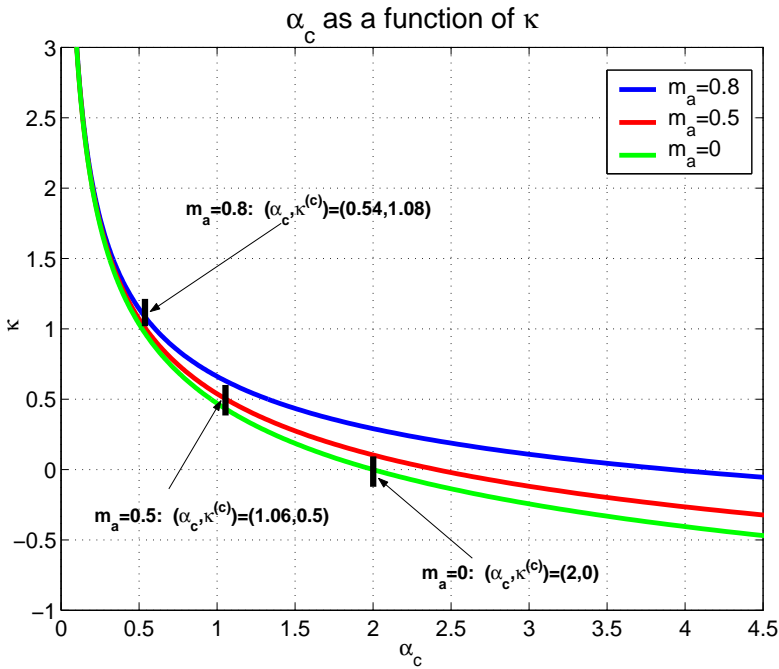,width=10.5cm,height=9cm}}
\caption{$\alpha_c$ as a function of $\kappa$}
\label{fig:alfackappa}
\end{figure}

\begin{figure}[htb]
\centering
\centerline{\epsfig{figure=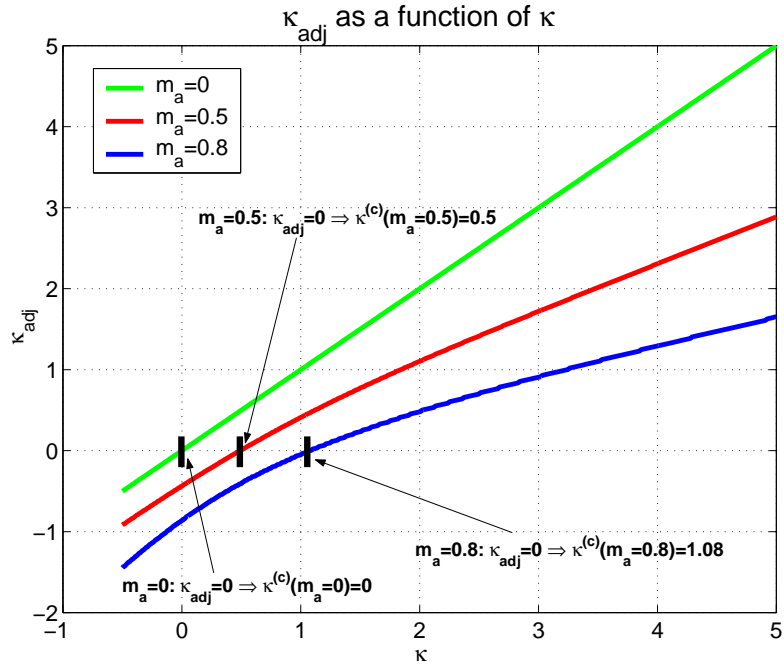,width=10.5cm,height=9cm}}
\caption{$\kappa_{adj}$ as a function of $\kappa$}
\label{fig:kappakappaadj}
\end{figure}

\section{Conclusion}
\label{sec:conc}

In this paper we revisited the so-called Gardner problem. The problem is one of the most fundamental/well-known feasibility problems and appears in a host areas, statistical physics, neural networks, integral geometry, combinatorics, to name a few. Here, we were interested in the so-called random spherical variant of the problem, often referred to as the random spherical perceptron.

Various features of the Gardner problem are typically of interest. We presented a framework that can be used to analyze the problem with pretty much all of its features. To give an idea how the framework works in practice we chose one of the perceptron features, called the storage capacity and analyzed it in details. We provided rigorous mathematical results for the values of storage capacity for certain range of parameters. We also proved that the results that we obtained are rigorous upper bounds on the storage capacity in the entire range of the thresholding parameter $\kappa$.

In addition to the standard uncorrelated version of the Gardner problem we also considered the correlated version and provided a set of results similar to those that we provided in the uncorrelated case. We again proved that the predictions obtained in \cite{Gar88} based on the statistical physics replica approach are at the very least the upper bounds for the value of the storage capacity and in certain range of the thresholding parameter actually the exact values of the storage capacity.

To maintain the easiness of the exposition throughout the paper we presented a collection of theoretical results for a particular type of randomness, namely the standard normal one. However, as was the case when we studied the Hopfield models in \cite{StojnicHopBnds10,StojnicMoreSophHopBnds10}, all results that we presented for the uncorrelated case can easily be extended to cover a wide range of other types of randomness. There are many ways how this can be done (and the rigorous proofs are not that hard either). Typically they all would boil down to a repetitive use of the central limit theorem. For example, a particularly simple and elegant approach would be the one of Lindeberg \cite{Lindeberg22}. Adapting our exposition to fit into the framework of the Lindeberg principle is relatively easy and in fact if one uses the elegant approach of \cite{Chatterjee06} pretty much a routine. However, as we mentioned when studying the Hopfield model \cite{StojnicHopBnds10}, since we did not create these techniques we chose not to do these routine generalizations. On the other hand, to make sure that the interested reader has a full grasp of a generality of the results presented here, we do emphasize again that pretty much any distribution that can be pushed through the Lindeberg principle would work in place of the Gaussian one that we used. When it comes to the correlated case the results again hold for a wide range of randomness, however one has to carefully account for the asymmetry of the problem.

It is also important to emphasize that we in this paper presented a collection of very simple observations. In fact the results that we presented are among the most fundamental ones when it comes to the spherical perceptron. There are many so to say more advanced features of the spherical perceptron that can be handled with the theory that we presented here. More importantly, we should emphasize that in this and a few companion papers we selected problems that we considered as classical and highly influential and chose to present the mechanisms that we developed through their analysis. Of course, a fairly advanced theory of neural networks has been developed over the years. The concepts that we presented here we were also able to use to analyze many (one could say a bit more modern) other problems within that theory (for example, various other dynamics can be employed, more advanced different network structures have been proposed and can be analyzed, and so on). However, we thought that before presenting how the mechanisms we created work on more modern problems, it would be in a sense respectful towards the early results created a few decades ago to first introduce our concepts through the classical problems. We will present many other results that we were able to obtain elsewhere.

\begin{singlespace}
\bibliographystyle{plain}
\bibliography{GardGenRefs}
\end{singlespace}

\end{document}